\documentclass[a4paper,11pt]{amsart}

\usepackage{hyperref}
\hypersetup{backref,colorlinks=true,allcolors=black}
\usepackage[foot]{amsaddr}
\usepackage{mathrsfs}
\usepackage{enumerate}
\usepackage{dsfont}
\usepackage{mathtools}
\usepackage{amssymb}
\usepackage{xcolor}

\newtheorem{prop}{Proposition}
\newtheorem{Lemma}{Lemma}
\newtheorem{definition}{Definition}

\newtheorem{theorem}{Theorem}
\newtheorem{remark}{Remark}
\newtheorem{example}{Example}

\newtheorem*{acknow*}{Acknowledgments}

\newcommand{\ignore}[1]{}
\newcommand{\R}{\mathbb{R}}
\newcommand{\N}{\mathbb{N}}
\newcommand{\X}{\mathcal{X}}
\newcommand{\Y}{\mathcal{Y}}
\newcommand{\W}{\mathcal{W}}
\newcommand{\Z}{\mathcal{Z}}

\newcommand{\G}{\mathcal{G}}
\newcommand{\Ro}{\mathcal{R}}
\newcommand{\on}[1]{\operatorname{#1}}
\newcommand{\norm}[1]{\left\lVert #1 \right\rVert}
\newcommand{\abs}[1]{\left\vert #1 \right\rvert}
\newcommand{\E}[1]{\mathbb{E}{\left[ #1\right]}}

\newcommand{\V}[1]{\mathrm{Var}{\left( #1\right)}}

\newcommand{\ceil}[1]{\lceil #1 \rceil}
\newcommand{\floor}[1]{\lfloor #1 \rfloor}

\DeclarePairedDelimiter\autobracket{(}{)}
\newcommand{\brac}[1]{\autobracket*{#1}}
\newcommand{\inner}[1]{\left\langle #1 \right\rangle}
\newcommand{\evalat}[1]{\left. #1 \right\rvert}

\theoremstyle{definition}
\newtheorem{proofcase}{Case}
\makeatletter
\@addtoreset{proofcase}{theorem}
\makeatother

\theoremstyle{definition}
\newtheorem{proofstep}{Step}
\makeatletter
\@addtoreset{proofstep}{theorem}
\makeatother

\begin{document}
	\bibliographystyle{amsalpha}
	
	
	\author{Solesne Bourguin$^{\dagger}$}
	\address{$^{\dagger}$Boston University, Department of Mathematics and Statistics,  111 Cummington Mall, Boston, MA 02215, USA}
	\email{bourguin@math.bu.edu}
	\author{Thanh Dang$^{\dagger}$}
	\email{ycloud77@bu.edu}
	
	\title[Non-stationary
	Gaussian correlated Wishart matrices]{High dimensional regimes of non-stationary
		Gaussian correlated Wishart matrices}
	
	\thanks{S. Bourguin was supported in part by the Simons Foundation
		grant 635136}
	
	\begin{abstract}
		We study the high-dimensional asymptotic regimes of correlated Wishart matrices
		$d^{-1}\Y \Y^T$, where $\Y$ is a $n\times d$ Gaussian random matrix
		with correlated and non-stationary entries. We prove that under
		different normalizations, two distinct regimes emerge
                as both $n$ and $d$ grow to infinity. The first
		regime is the one of central
		convergence, where the law of the properly renormalized Wishart matrices
		becomes close in Wasserstein distance to that of a Gaussian orthogonal ensemble matrix. In the
		second regime, a non-central convergence happens, and the law of the
		normalized Wishart matrices becomes close in Wasserstein distance to that of the so-called
		Rosenblatt-Wishart matrix recently introduced by Nourdin and Zheng. We
		then proceed to show that the convergences stated above also hold in a
		functional setting, namely as weak convergence in
		$C([a,b];M_n(\R))$.
As an application of our main result (in
                the central convergence regime), we show that it can
                be used to prove convergence in expectation of the
                empirical spectral distributions of the Wishart
                matrices to the semicircular law.
                Our findings complement and extend a rich
		collection of results on the study of the fluctuations of Gaussian Wishart
		matrices, and we provide explicit examples based on Gaussian entries
		given by normalized increments of a bi-fractional or a sub-fractional
		Brownian motion. 
	\end{abstract}
	
	
	\maketitle
	
	\section{Introduction and main results}
	\label{introandmainresults}
	
	\noindent Random matrix theory plays a fundamental role in many areas
	of mathematics, either theoretical ones such as non-commutative algebra,
	combinatorics, geometry or spectral analysis, or applied ones such as
	statistical physics, signal processing or multivariate analysis and
	statistical theory. In the latter, one type of random matrices are
	particularly important as they are used to model, for example, sample
	covariance matrices (see e.g., the surveys \cite{BMN,IJ,RasWi}), which in the era of data driven
analysis have now a growing importance in practice. This type of matrices
	are called Wishart matrices and have been introduced by the
	statistician John Wishart in \cite{Wis}. Given an underlying $n\times
	d$ random matrix $\Y$, the associated Wishart matrix is given by
	$d^{-1}\Y \Y^T$ and is hence a symmetric $n\times n$ random matrix. In
	this paper, we are interested in the asymptotic behaviour of such
	large Wishart matrices as both dimensions grow to infinity. The fact
	that both dimensions are allowed to grow in our context is
	fundamental, especially when one interprets the growth of the
	dimensions as underlying data sets becoming very large over time if
	the Wishart matrices are seen as being sample covariance matrices for
	instance. The case where $n$ is taken to be fixed and only $d$ grows
	to infinity (called the one-dimensional regime) is well understood via standard probabilistic results such
	as the laws of large numbers, but this setting represents nowadays a
	drawback and is not realistic enough in applications when considering how common it
	has become to have increasingly bigger data sets that grow with time
	as one continuously collects new data. The high-dimensional regime,
	i.e., the case where both $n$ and $d$
	grow to infinity, possibly at different paces, is much more difficult to
	apprehend and has triggered many different and complementary lines of
	work in the recent years (see e.g.,
	\cite{BDT,BBN,Bu,BG,Chatter,FK,JL,MarPas,Mik,NZ,RaczRic}).
	\\~\\
	One particular setting
	that has
	received more attention than others is the one where the matrix $\Y$
	has Gaussian entries, which is the most common case to arise in
	applications. More specifically, the case where the Gaussian entries
	are increments of Gaussian processes (such as a Brownian or fractional
	Brownian motion) is of particular interest when considering the
	modeling of systems that evolve in time. When all the entries of $\Y$
	are considered to be independent and coming from a stationary Gaussian
	process, the fluctuations of the associated Wishart matrix have been
	studied in \cite{Bu,BG}. The assumption of full independence of the
	entries has then been relaxed in \cite{NZ} where the authors
        assume either
	row independence, with a possible correlation in each row
        separately, or overall correlation (depending on the setting),
	but keeping the assumption of stationarity of the underlying Gaussian
	process.
	\\~\\
	Our goal in this paper is to study the high-dimensional
	fluctuations of Wishart matrices based on Gaussian non-stationary
	entries with a self-similarity property, and hence relax
	the stationary assumption made in \cite{NZ} to allow for Gaussian
	entries coming from non-stationary Gaussian processes, such as
	bi-fractional or sub-fractional Brownian motions. We use the class of
	processes introduced in \cite{HN} and further studied in \cite{CNN} characterized by the fact that the
	covariance function of the Gaussian processes that are members of this class satisfy hypotheses \textbf{(H.1)} and
	\textbf{(H.2)} stated below (the covariance
	function of the processes is a perturbation of the covariance function
	of a fractional Brownian motion, and includes among other important examples, bi-fractional
	Brownian motion, or sub-fractional Brownian motion).
	\\~\\
	We also address a question that has not
	been studied so far in this context to the best of our knowledge,
	namely that of functional convergence when the Wishart matrices are
	seen as matrix-valued processes. Deriving functional versions of limit
	theorems and convergence results is of utmost interest, especially
	when considering systems that naturally describe phenomenons evolving
	in time, as illustrated by the fast-growing literature on this topic
	(see for instance \cite{CNN,NourNua} for functional limit theorems
	related to the celebrated Breuer-Major theorem, \cite{Ba,bourcamp} for
	a quantitative approach based on Stein's method in Banach spaces,
	among many other references). Our (non-functional) convergence
	results ensuring the convergence in the sense of finite-dimensional
	distributions, we prove that the sequences of Wishart matrices we
	consider are tight in $C([a,b];M_n(\R))$. We consider the indexing
	parameter of such matrix-valued processes to be part of the $d$
	dimension in the form of a dependency of this dimension on it by
	replacing $d$ by $\floor{dx}$, where $x$ is the indexing parameter of
	the matrix process. An applied way of looking at this setting is to
	interpret $x$ as time, and considering that the $d$ dimension grows
	continuously with time, which is a very natural setting in many
	applications (such as financial time series or temperature readings
	for instance).
	\\~\\
As an application of our results, we prove the convergence in expectation
of the empirical spectral distributions of these Wishart matrices to
the semicircular law (in the case where the Wishart matrices exhibit a
central limit behavior), which is a central question in random matrix
theory. Our methodology to derive this result complements the more
classical methods such as the characteristic function approach, the
method of moments or the Stieltjes transform.
\\~\\
	Let us describe the class of processes we will be working
	with. Let $X$ be a member of the class of self-similar processes
	introduced in \cite{HN}. For reference, a stochastic process
        $\left\{ X_s \colon s \geq 0 \right\}$ is called a
        self-similar process with self-similarity parameter $H>0$ if for all $c>0$,
	\begin{align*}
	    \left\{X\brac{cs}\colon s \geq 0 \right\}\stackrel{\text{dist}}{=}\left\{c^H X\brac{s}\colon s \geq 0 \right\},
	\end{align*}
        where $\stackrel{\text{dist}}{=}$ denotes equality in
        distribution. In our case, $\left\{ X_s \colon s \geq 0 \right\}$
	is a centered, self-similar Gaussian process with self-similarity
	parameter $\beta \in (0,1)$. Define $\phi \colon [1,\infty) \to \R$ by
	$\phi(x)=\E{X_1 X_x}$, so that, for $0 < s \leq t$, we have 
	\begin{equation*}
		\E{X_s X_t}=s^{2\beta}\E{X_1 X_{\frac{t}{s}}} = s^{2\beta}\phi \left( \frac{t}{s} \right).
	\end{equation*}
	Hence, $\phi$ characterizes the covariance function of $X$. Moreover,
	the following two assumptions are assumed to hold for all members of
	this class of processes, and were both introduced in \cite{HN}.
	\begin{enumerate}
		\item[\textbf{(H.1)}] There exists $\alpha \in (0,2\beta]$ such that
		$\phi$ has the form 
		\begin{equation*}
			\phi(x) =  -\lambda (x-1)^{\alpha} + \psi(x),
		\end{equation*}
		where $\lambda >0$ and $\psi$ is twice-differentiable on an open set
		containing $[1,\infty)$ and there exists a constant $C \geq 0$ such
		that, for any $x \in (1,\infty)$,
		\begin{itemize}
			\item[(a)] $\abs{\psi'(x)}\leq Cx^{\alpha-1}$
			\item[(b)] $\abs{\psi''(x)}\leq Cx^{-1}(x-1)^{\alpha-1}$
			\item[(c)] $\psi'(1)=\beta \psi(1)$ when $\alpha\geq 1$.
		\end{itemize}
	\end{enumerate}
	\begin{enumerate}
		\item[\textbf{(H.2)}] There are constants $C >0$ and $1 < \nu
		\leq 2$ such that, for all $x \geq 2$,
		\begin{itemize}
			\item[(d)] $\abs{\phi'(x)}= \begin{cases} 
			C(x-1)^{-\nu} & \alpha< 1, \\
			C(x-1)^{\alpha-2} & \alpha\geq 1.
			\end{cases}$
			\item[(e)] $\abs{\phi''(x)}= \begin{cases} 
			C(x-1)^{-\nu-1} & \alpha< 1, \\
			C(x-1)^{\alpha-3} & \alpha\geq 1.
			\end{cases}$
		\end{itemize}
	\end{enumerate}
	The reader is referred to \cite[Section 4]{HN} for worked out examples
	of Gaussian processes satisfying assumptions \textbf{(H.1)} and
	\textbf{(H.2)}, among which, as pointed out earlier, the bi-fractional
	Brownian motion and the sub-fractional Brownian motion.
	\\~\\
	Now, for $k \geq 0$, define 
	\begin{equation*}
		\Delta X_k = X_{k+1}-X_{k}\quad\mbox{and}\quad Y_{k}=\frac{\Delta X_{k}}{\norm{\Delta X_{k}}_{L^2{(\Omega)}}}.
	\end{equation*}
	We are now ready to introduce the Gaussian random matrices $\Y$ our Wishart
	matrices will be built upon. Let $\left\{ X^i \colon i \in \N
	\right\}$ be i.i.d. copies of $X$ and write 
	\begin{equation*}
		\Delta X^i_k=X^i_{k+1}-X^i_{k} \quad\mbox{and}\quad Y^i_{k}=\frac{\Delta X^i_{k}}{\norm{\Delta X^i_{k}}_{L^2{(\Omega)}}}.
	\end{equation*}
	For any $x\in [a,b]$ where $a<b$ are two positive constants, let $\Y$ be a $n\times \floor{dx}$ matrix with
	entries given by $Y^i_k$, $1 \leq i \leq n$, $1 \leq k \leq
	\floor{dx}$. Whenever the parameter $\alpha$ (appearing in
	\textbf{(H.1)}) of the process $X$ is such that $0<\alpha<\frac{3}{2}$, we define the
	Wishart matrix $\W_{n,\floor{dx}}$ to be
	\begin{align*}
		\W_{n,\floor{dx}}=\frac{\floor{dx}}{\sqrt{d}}\brac{\frac{1}{\floor{dx}}\Y\Y^T-I}.
	\end{align*}
	$\W_{n,\floor{dx}}$ is a $n\times n$ matrix with entries given by, for
	any $1 \leq i,j \leq n$, 
	\begin{equation*}
		W_{ij}(\floor{dx})=\frac{1}{\sqrt{d}}\sum_{k=1}^{\floor{dx}}\left(Y^i_kY^j_k -\mathds{1}_{\left\{ i=j \right\}}\right).
	\end{equation*}	
	Whenever $\alpha=\frac{3}{2}$, we define the Wishart matrix
	$\widetilde{\W}_{n,\floor{dx}}$, which differs from
	$\W_{n,\floor{dx}}$ by the normalization of its entries, by 
	\begin{align*}
		\widetilde{\W}_{n,\floor{dx}}=\frac{\floor{dx}}{\sqrt{d}\ln d}\brac{\frac{1}{\floor{dx}}\Y\Y^T-I}.
	\end{align*}
	Finally, whenever $\alpha >\frac{3}{2}$, we define another version of
	the Wishart matrix (with yet another normalization of the entries) by
	\begin{equation}
		\label{defrosmat}
		\widehat{\W}_{n,\floor{dx}}=\frac{\floor{dx}}{d^{\alpha-1}}\brac{\frac{1}{\floor{dx}}\Y\Y^T-I}.
	\end{equation}
	\begin{remark}
		The Wishart matrices introduced above, $\W_{n,\floor{dx}}$,
		$\widetilde{\W}_{n,\floor{dx}}$ and $\widehat{\W}_{n,\floor{dx}}$
		are essentially the same object, only differing by the normalization
		of their entries. The fact that several normalization are required
		depending on the value of the parameter $\alpha$ corresponds to the
		different asymptotic regimes appearing depending on said values.
	\end{remark}
	\begin{remark} We let the parameter $x\in [a,b]$ in order to
          study functional convergence of Wishart matrices as
          $d$ grows to infinity. If functional convergence is not the topic of
          interest for applications, nothing prevents one from taking
          $x =1$ to be fixed and recover a classical $n \times d$
          matrix. The assumption $a>0$ allows us to sidestep the case $x=0$ and ensure $\Y$ of size $n\times \floor{dx}$ is well-defined, as long as $d$ is sufficiently large. 
	\end{remark}
	\noindent From now on, whenever $x$ is considered to be fixed
	(Sections \ref{centralconvergencesection}, \ref{threehalfsection} and
	\ref{rosenblattconvergencesection}), we will drop the $x$ dependency in
	our notation and write, for example, $W_{ij}$ in place of
	$W_{ij}(\floor{dx})$. Moreover, in what follows, $C$ denotes a generic
	positive constant that may vary from line to line.
	\\~\\
	Our first main result establishes central convergence in the case
	where $0<\alpha<\frac{3}{2}$. The notation $d_W$ stands for the
	Wasserstein distance introduced and defined in Section \ref{sectionbackground}. 	
	
	\begin{theorem}
		\label{theoremcentralconvergence}
		Let $0<\alpha<\frac{3}{2}$. Then, the Wishart matrix
		$\W_{n,\floor{dx}}$ is close to the Gaussian Orthogonal Ensemble matrix
		$\Z_n$ (in finite-dimensional distribution) defined to be a $n\times
		n$ symmetric matrix with independent entries such that $Z_{ii}\sim
		N(0,2\sigma^2)$ and $Z_{ij}\sim N(0,\sigma^2)$ for $i\neq j$, where
		$\sigma^2$ is defined in Lemma \ref{varianceLemma}. Furthermore, the
		following quantitative bound holds
		\begin{equation*}
			d_{W}\brac{\W_{n,\floor{dx}},\Z_n}\leq C\brac{n^\frac{3}{2}r(\alpha,\nu)+nd^{2\alpha-3}+nd^{-1}},
		\end{equation*}
		where
		\begin{equation*}
			r(\alpha,\nu)= \begin{cases} 
				d^{\frac{2\alpha-3}{2(9-2\alpha)}} &\text{if } \alpha<1 \mbox{ and } \alpha+\nu<2 \\
				d^{-\frac{1}{2}} &\text{if } \alpha<1 \mbox{ and } \alpha+\nu\geq 2 \\
				d^{-\frac{1}{2}} &\text{if } 1\leq \alpha<\frac{5}{4}\\
				d^{-\frac{1}{2}}(\ln{d})^\frac{3}{2} &\text{if } \alpha=\frac{5}{4}\\
				d^{2\alpha-3} &\text{if } \frac{5}{4}< \alpha<\frac{3}{2}
			\end{cases}.
		\end{equation*}
	\end{theorem}	
	\begin{remark}
		Since  we have
		$ -\frac{1}{2}< \frac{2\alpha-3}{2(9-2\alpha)}<0$ for
		$\alpha<1$, the above convergence rate in the case where $\alpha<1$ and $\alpha+\nu<2$ is weaker than the rate appearing in the stationary case treated in \cite{NZ}. This comes from the fact that some estimates used in \cite{NZ} are not valid in our increased generality, and other
		arguments are needed, giving rise to this different
		convergence rate. One can hence see our result as being complementary to
		those in \cite{NZ} as their rate is better if the process $X$ is
		stationary, but ours allows to cover the non-stationary case as
		well. It is not surprising that the fact that our result accommodates
		many more processes than the stationary ones comes at the price of a
		slightly less optimal rate. From a technical point of view, it comes
		from the fact that the estimates in Lemma \ref{fbmcomparisonLemma} do
		not hold for all instances of $\alpha$ and $\beta$, and warrants new
		estimates.
	\end{remark}
	\noindent Before going any further, we would like to illustrate the
	results of Theorem \ref{theoremcentralconvergence} on two interesting examples
	of processes that are not stationary (and hence not covered by, for
	instance, \cite{NZ}), namely the bi-fractional Brownian motion and the
	sub-fractional Brownian motion.
	\begin{example}
		\label{remarkcentralconvergence}
		If $X$ is a bi-fractional Brownian motion, i.e., a centered Gaussian
		process with covariance function given by 
		\begin{equation*}
			\E{X_tX_s} = 2^{-K} \left( \left( t^{2H} + s^{2H} \right)^K - \abs{t-s}^{2HK} \right),
		\end{equation*}
		where $H\in (0,1)$, $K\in [a,b]$, then it was derived in
		\cite[Section 4.1]{HN} that $\alpha=2\beta=2HK$ and $\nu=(1+2H-2HK)\wedge (2-2HK)$. Theorem
		\ref{theoremcentralconvergence} then yields
		\[d_{W}\brac{\W_{n,\floor{dx}},\Z_n}\leq  C\begin{cases} 
		n^\frac{3}{2}d^{\frac{4HK-3}{2(9-4HK)}} &\text{if }
		2HK<1 \mbox{ and } 2H< 1 \\
		n^\frac{3}{2}d^{-\frac{1}{2}}&\text{if } 2HK<1 \mbox{ and } 2H>1 \\
		n^\frac{3}{2}d^{-\frac{1}{2}}+nd^{4HK-3} &\text{if } 1\leq 2HK<\frac{5}{4}\\
		n^\frac{3}{2}d^{-\frac{1}{2}}(\ln{d})^\frac{3}{2} &\text{if } 2HK=\frac{5}{4}\\
		n^\frac{3}{2}d^{4HK-3} &\text{if } \frac{5}{4}< 2HK<\frac{3}{2}
		\end{cases}.
		\]
		\\~\\
		\noindent If $X$ is a sub-fractional Brownian motion, i.e., a centered Gaussian
		process with covariance function given by 
		\begin{equation*}
			\E{X_tX_s} = t^{2H} + s^{2H} -\frac{1}{2}\left( (t+s)^{2H} + \abs{t-s}^{2H} \right),
		\end{equation*}
		where $H\in (0,1)$, then it was derived in
		\cite[Section 4.2]{HN} that $\alpha=2\beta=2H$ and $\nu=2-2H$. Theorem
		\ref{theoremcentralconvergence} then yields
		\begin{equation*}
			d_{W}\brac{\W_{n,\floor{dx}},\Z_n}\leq  C\begin{cases} 
				n^\frac{3}{2}d^{-\frac{1}{2}} &\text{if } 0\leq H<\frac{1}{2}\\
				n^\frac{3}{2}d^{-\frac{1}{2}}+ nd^{4H-3} &\text{if } \frac{1}{2}\leq H<\frac{5}{8}\\
				n^\frac{3}{2}d^{-\frac{1}{2}}(\ln{d})^\frac{3}{2} &\text{if } H=\frac{5}{8}\\
				n^\frac{3}{2}d^{4H-3} &\text{if } \frac{5}{8}< H<\frac{3}{4}
			\end{cases}
		\end{equation*}
		which are the same convergence rates obtained in \cite{NZ} for
		fractional Brownian motion.
	\end{example}
	\noindent When $\alpha=\frac{3}{2}$, we still have central
	convergence, but under a different normalization of the entries of the
	Wishart matrix, hence giving rise to a different regime for central
	convergence when compared to the case $\alpha<\frac{3}{2}$.
	\begin{theorem}
		\label{theoremthreehalf}
		Let $\alpha=\frac{3}{2}$. Then, the Wishart matrix
		$\widetilde{\W}_{n,\floor{dx}}$ is close to the Gaussian Orthogonal Ensemble matrix
		$\widetilde{\Z}_n$ (in finite-dimensional distribution) defined to be a $n\times
		n$ symmetric matrix with independent entries such that $Z_{ii}\sim
		N(0,2\rho^2)$ and $Z_{ij}\sim N(0,\rho^2)$ for $i\neq j$, where
		$\rho^2$ is defined in Lemma \ref{varianceLemmathreehalf}. Furthermore, the
		following quantitative bound holds
		\begin{align*}
			d_{W}\brac{\widetilde{\W}_{n,\floor{dx}},\widetilde{\Z}_n}\leq C\frac{n^{3/2}}{\ln d}.
		\end{align*}
	\end{theorem}
	\noindent In the case where $\frac{3}{2}< \alpha<2$, convergence
	still happens, but it is not central anymore. This exhibits another
	asymptotic regime, both in terms of normalization of the entries, as
	well as in terms of the nature of the limit. Before stating our
	result, let us define the limiting object that features in it.
	\begin{definition}
		\label{rosenblattwishartmatrixdef}
		The $n \times n$ Rosenblatt-Wishart matrix is the random symmetric
		matrix $\Ro_n$ with its entries given by $R_{ij} = \lim_{d \to
			\infty} \widehat{W}_{ij}$, for all $1 \leq i,j \leq n$. This limit
		always exists and is well-defined, as ensured by Lemma
		\ref{Lemmal2convergence}, and the entries of $\Ro_n$ are elements of
		the second Wiener chaos associated to $X$.
	\end{definition}
	\noindent The following theorem describes the non-central asymptotic regime
	where the above defined object appears as the limit.
	\begin{theorem}
		\label{theoremrosenblatt}
		Let $\frac{3}{2}<\alpha<2$. Then, the Wishart matrix
		$\widehat{\W}_{n,\floor{dx}}$ is close to the Rosenblatt-Wishart
		matrix $\Ro_n$ (in finite-dimensional distribution) defined in
		Definition \ref{rosenblattwishartmatrixdef}. Furthermore, the
		following quantitative bound holds
		\begin{align*}
			d_{W}\brac{\widehat{\W}_{n,\floor{dx}},\Ro_n}\leq  Cnd^{\frac{3-2\alpha}{2}}.
		\end{align*}
	\end{theorem}
	\begin{example}
		If $X$ is a bi-fractional Brownian motion as defined in Example \ref{remarkcentralconvergence} with $\frac{3}{4}<HK<1$, then $\widetilde{\W}_{n,\floor{dx}}$ converges to a Rosenblatt-Wishart matrix at the rate $nd^{\frac{3-2\alpha}{2}}$. The same limit distribution and convergence rate apply  when $X$ is a sub-fractional Brownian motion with $\frac{3}{4}<H<1$.
		\\~\\		
		Comparing our result to \cite{NZ}, Nourdin and Zheng obtain the same convergence rate when $X$ is a fractional Brownian motion with $\frac{3}{4}<H<1$. Here we note that $\alpha=2\beta=2H$ for a fractional Brownian motion.
	\end{example}
	\noindent As functional limit theorems are taking an increasingly growing
	importance in the literature, it is a natural question to ask whether
	the convergence results we have obtained so far (regarding the finite
	dimensional distributions of the Wishart matrices under consideration)
	can be made functional under potentially additional assumptions. We
	have chosen to explore the case where the second dimension of the
	matrix $\Y$ grows continuously. One could think of the index $x$
	appearing in the second dimension as time for instance and consider
	the case where the second dimension of $\Y$ grows with time (because
	it is continuously being fed new data over time for example). Our
	next result strengthens and complements Theorems \ref{theoremcentralconvergence}, \ref{theoremthreehalf} and
	\ref{theoremrosenblatt} by settling the question of functional convergence in $C([a,b];M_n(\R))$.	
	\begin{theorem}
		\label{theoremfunctionalnonquantitative}
		The convergences stated in Theorems \ref{theoremcentralconvergence}, \ref{theoremthreehalf} and \ref{theoremrosenblatt} hold in $C\brac{[a,b];M_n(\R)}$. 
              \end{theorem}
              ~\\
\noindent We now present an application of Theorems \ref{theoremcentralconvergence} and \ref{theoremthreehalf} to proving
that the empirical spectral distributions of the Wishart matrices
$\frac{1}{\sqrt{n}}\W_{n,\floor{dx}}$ and $\frac{1}{\sqrt{n}}\widetilde{\W}_{n,\floor{dx}}$ converge in
expectation to the semicircular distribution. Recall that for a $n
\times n$ symmetric matrix $\mathcal{M}_n$, the (normalized) empirical
spectral distribution is defined as 
\begin{equation*}
\mu_{\frac{1}{\sqrt{n}}\mathcal{M}_n} = \frac{1}{n}\sum_{i=1}^n \delta_{\lambda_i(\mathcal{M}_n)/\sqrt{n}},
\end{equation*}
where $\lambda_1(\mathcal{M}_n) \leq \cdots \leq
\lambda_n(\mathcal{M}_n)$ are the (real) eigenvalues of
$\mathcal{M}_n$, counting multiplicity. Recall also that the semicircular distribution
$\nu_t$ with variance $t>0$ is a probability distribution defined by
	\begin{align*}
	\nu_t(dx) = \frac{1}{2\pi t}\sqrt{(4t-x^2)_+}dx.
	\end{align*}
By Wigner's semicircle law, we know that the empirical spectral distribution of the GOE matrices
$\Z_n$ defined in Theorem \ref{theoremcentralconvergence} converges in
expectation to $\nu_{\sigma^2}$ (where $\sigma^2$ is the variance of the entries of the GOE matrices $\Z_n$ from Theorem \ref{theoremcentralconvergence}). The following result is an
application of Theorem \ref{theoremcentralconvergence} and highlights the spectral behavior of the class of Wishart matrices studied in this paper. 
\begin{theorem}
	\label{theorem_app}
		The empirical spectral distribution of $\frac{1}{\sqrt{n}}\W_{n,\floor{dx}}$
                converges in expectation to the semicircular
                distribution $\nu_{\sigma^2}$. In other words, as $n,d
                \to\infty$, it holds that
		\begin{align*}
		\E{\mu_{\frac{1}{\sqrt{n}}\W_{n,\floor{dx}}}}\longrightarrow \nu_{\sigma^2}.
		\end{align*}
              \end{theorem}
              \begin{remark}
We emphasize that the above theorem is just one possible application
of our main results, as it focuses on one particular statistic of the
the Wishart matrices, namely the empirical spectral distribution. The link between the spectral statistics of
Wishart and Wigner matrices, the latter being known as the Gaussian
Orthogonal Ensemble if the entries are Gaussian, has been studied
extensively. Tracy and Widom obtained the limiting distribution of the
largest eigenvalue of Wigner matrices in \cite{Tra,Tra2}, and it is
now known as the Tracy-Widom law. Johnstone \cite{John} and El-Karoui
\cite{Karo} obtained the same limit distribution for largest
eigenvalues of real and complex Wishart matrices under the regime
$\frac{d}{n}\to c\in [0,\infty]$. More recent work on the transition
from Wishart to Wigner matrices in the high dimensional setting and/or
the corresponding transition of spectral statistics such as condition number, extremal eigenvalues and others includes \cite{Bu,BG,CW,JL,RaczRic}. In this context, Theorem \ref{theorem_app} once again demonstrates how similar the spectrum of Wishart and Wigner matrices behave, even if the independence condition between the entries of $\Y$ is relaxed. 
\end{remark}
\begin{proof}[Proof of Theorem \ref{theorem_app}]
We need to prove that for any fixed $k \geq 1$, the $k$-th moment
\begin{equation*}
\frac{1}{n}\E{\on{Tr}\left(\left(\frac{1}{\sqrt{n}}\W_{n,\floor{dx}}\right)^k \right)}
\end{equation*}
converges to the $k$-th moment of the semicircular distribution
$\nu_{\sigma^2}$. By Theorem \ref{theoremcentralconvergence}, we have,
as $d \to \infty$,
\begin{equation*}
\frac{1}{\sqrt{n}}\W_{n,\floor{dx}}  \longrightarrow \frac{1}{\sqrt{n}}\Z_n,
\end{equation*}
where the convergence holds in distribution. By the continuous mapping
theorem, it follows that  
\begin{equation}
  \label{convindistribtogoe}
\frac{1}{n}\on{Tr}\left(\left(\frac{1}{\sqrt{n}}\W_{n,\floor{dx}}\right)^k \right)  \longrightarrow \frac{1}{n}\on{Tr}\left(\left(\frac{1}{\sqrt{n}}\Z_n\right)^k \right),
\end{equation}
as $d \to \infty$, where the convergence holds in distribution, due to
the fact that the map $\mathcal{M}_n \mapsto \frac{1}{n}\on{Tr}\left(\mathcal{M}_n^k
\right)$ is continuous as a multivariate polynomial. Now, the fact that the
sequence $$\left\{
  \frac{1}{n}\on{Tr}\left(\left(\frac{1}{\sqrt{n}}\W_{n,\floor{dx}}\right)^k
  \right) \colon d \geq 1 \right\}$$ is uniformly integrable (by the
hypercontractivity of the Wiener chaos and the fact that the entries
of $\W_{n,\floor{dx}}$ have bounded variances -- see Lemma
\ref{varianceLemma}) together with the convergence in distribution
\eqref{convindistribtogoe} yields
\begin{equation*}
\frac{1}{n}\E{\on{Tr}\left(\left(\frac{1}{\sqrt{n}}\W_{n,\floor{dx}}\right)^k
  \right)} \longrightarrow \frac{1}{n}\E{\on{Tr}\left(\left(\frac{1}{\sqrt{n}}\Z_n\right)^k.
  \right)}
\end{equation*}
Letting $n \to \infty$ and invoking Wigner's semicircle law now yields the desired fact that the $k$-th moment
\begin{equation*}
\frac{1}{n}\E{\on{Tr}\left(\left(\frac{1}{\sqrt{n}}\W_{n,\floor{dx}}\right)^k \right)}
\end{equation*}
converges (as first $d$, then $n$ go to infinity) to the $k$-th moment of the semicircular distribution
$\nu_{\sigma^2}$, which concludes the proof.
\end{proof}
\begin{remark}
Theorem \ref{theorem_app} is stated for the matrices
$\W_{n,\floor{dx}}$, but the same result holds for the matrices
$\widetilde{\W}_{n,\floor{dx}}$ appearing in Theorem \ref{theoremthreehalf} with no modifications
of the proof, other than the limiting semicircular distribution being
$\nu_{\rho^2}$ in this case (and invoking Theorem \ref{theoremthreehalf} instead of Theorem
\ref{theoremcentralconvergence}).
\end{remark}

	\noindent This paper is organized as follows. Section \ref{sectionbackground}
	provides the needed elements of Malliavin calculus, as well as some
	results related to Stein's method. Section
	\ref{centralconvergencesection} is dedicated to the preparation of the
	proof and the proof of Theorem \ref{theoremcentralconvergence},
	Section \ref{threehalfsection} contains the proof of Theorem
	\ref{theoremthreehalf} while Section
	\ref{rosenblattconvergencesection} addresses the proof of Theorem
	\ref{theoremrosenblatt}. The proof of the functional version of our
	results (Theorem \ref{theoremfunctionalnonquantitative}) is given in
	Section \ref{nonquantitativefunctionalsection}. Section
	\ref{sectionLemma} gathers technical and auxiliary results needed
	for the proofs of the main results.

	\section{Preliminaries}
	\label{sectionbackground}
	\subsection{Overview of Malliavin calculus}
	Let $\frak{H}$ be a real separable Hilbert space and $\left\{ Z(h)
	\colon h \in \frak{H} \right\}$ an isonormal Gaussian process
	indexed by it, that is, a centered Gaussian family of random variables
	such that $\E{Z(h)Z(g)} = \left\langle h,g \right\rangle_{\frak{H}}$. Denote by  $I_{n}$ the multiple Wiener (or Wiener-It\^o) stochastic
	integral  of order $n \geq 0$ with respect to
	$Z$ (see \cite[Section 1.1.2]{Nua}). The mapping $I_{n}$ is an
	isometry between the Hilbert space $\frak{H}^{\odot n}$ (symmetric
	tensor product) equipped with the scaled norm
	$\frac{1}{\sqrt{n!}}\norm{\cdot}_{\frak{H}^{\otimes n}}$ and the
	Wiener chaos of order $n$, which is defined as the closed linear span
	of the random variables $$\left\{ H_n(Z(h)) \colon h \in \frak{H},\
	\norm{h}_{\frak{H}}=1 \right\},$$ where $H_{n}$ is the $n$-th Hermite
	polynomial given by $H_{0}=1$ and for $n\geq 1$,
	\begin{equation*}
		H_{n}(x)=\frac{(-1)^{n}}{n!} \exp \left( \frac{x^{2}}{2} \right)
		\frac{d^{n}}{dx^{n}}\left( \exp \left( -\frac{x^{2}}{2}\right)
		\right), \quad x\in \mathbb{R}.
	\end{equation*}
	Multiple Wiener integrals enjoy the following isometry property: for
	any integers $m,n \geq 1$,
	\begin{equation*}
		\E{I_{n}(f) I_{m}(g)} = \mathds{1}_{\left\{ n=m \right\}} n! \langle \tilde{f},\tilde{g}\rangle _{\frak{H}^{\otimes n}},
	\end{equation*}
	where $\tilde{f} $ denotes the symmetrization of $f$ and we recall that $I_{n}(f) = I_{n} ( \tilde{f} )$.
	\\~\\
	Recall the multiplication formula satisfied by multiple Wiener
	integrals: for any $n,m \geq 1$, and any $f\in \frak{H}^{\odot n}$ and
	$g\in \frak{H}^{\odot m}$, it holds that
	\begin{equation}
		\label{prodformula}
		I_n(f)I_m(g)= \sum _{r=0}^{n\wedge m} r! \binom{n}{r}\binom{m}{r} I_{m+n-2r}(f\otimes _r g),
	\end{equation}
	where the $r$-th contraction of $f$ and $g$ is defined by, for $0\leq r\leq m\wedge n$, 
	\begin{equation*}
		f\otimes _r g = \sum_{i_1,\ldots , i_r =1}^{\infty} \left\langle f,
		e_{i_1}\otimes \cdots \otimes e_{i_r}
		\right\rangle_{\frak{H}^{\otimes r}} \otimes \left\langle g,
		e_{i_1}\otimes \cdots \otimes e_{i_r}
		\right\rangle_{\frak{H}^{\otimes r}},
	\end{equation*}
	with $\left\{ e_i \colon i \geq 1 \right\}$ denoting a
	complete orthonormal system in $\frak{H}$. 
	\\~\\
	Recall that any square integrable random variable $F$ which is measurable with respect to the $\sigma$-algebra generated by $Z$ can be expanded into an orthogonal sum of multiple Wiener integrals:
	\begin{equation}
		\label{sum1} F=\sum_{n=0}^\infty I_{n}(f_{n}),
	\end{equation}
	where $f_{n}\in \frak{H}^{\odot n}$ are (uniquely determined)
	symmetric functions and $I_{0}(f_{0})=\mathbb{E}\left(  F\right)$.
	\\~\\
	Let $L$ denote the Ornstein-Uhlenbeck operator, whose action on a
	random variable $F$ with chaos decomposition \eqref{sum1} and such
	that $\sum_{n=1}^{\infty} n^{2}n! \norm{f_n}^2_{\frak{H}^{\otimes n}}<\infty$ is given by
	\begin{equation*}
		LF=-\sum_{n=1}^{\infty} nI_{n}(f_{n}).
	\end{equation*}
	A pseudoinverse $L^{-1}$ can be introduced via spectral calculus as
	follows:
	\begin{equation*}
		L^{-1}F=-\sum_{n=1}^{\infty} \frac{1}{n}I_{n}(f_{n}).
	\end{equation*}
	It follows that
	\begin{equation*}
		LL^{-1}F = F - \E{F}.
	\end{equation*}
	For $p>1$ and $k \in \mathbb{R}$, the Sobolev-Watanabe spaces
	$\mathbb{D}^{k ,p }$ are defined as the closure of
	the set of polynomial random variables with respect to the norm
	\begin{equation*}
		\Vert F\Vert _{k , p} =\Vert (I -L) ^{\frac{k }{2}} F \Vert_{L^{p} (\Omega )},
	\end{equation*}
	where $I$ denotes the identity operator. We denote by $D$ the Malliavin
	derivative that acts on smooth random variables of the form $F=g(Z(h_1),
	\ldots , Z(h_n))$, where $g$ is a smooth function with compact support
	and $h_i \in \frak{H}$, $1 \leq i \leq n$. Its action on such a random
	variable $F$ is given by
	\begin{equation*}
		DF=\sum_{i=1}^{n}\frac{\partial g}{\partial x_{i}}(Z(h_1), \ldots , Z(h_n)) h_{i}.
	\end{equation*}
	The operator $D$ is closable and continuous from $\mathbb{D}^{k ,
		p} $ into $\mathbb{D} ^{k -1, p} \left( \frak{H}\right)$.
	\\~\\
	Denote by $\delta$ the adjoint of $D$, which is known as the
	divergence operator. An element $u\in L^2(\Omega,\mathfrak{H})$
	belongs to $\on{dom}(\delta)$ only if there exists a constant $C_u$
	depending only on $u$ such that
	\begin{align*}
		\abs{\E{\inner{DF,u}_\mathfrak{H}}}\leq C_u\norm{F}_{L^2(\Omega)}
	\end{align*}
	for any $F\in \mathbb{D}^{1,2}$. In this case, we have the
	following integration by parts formula (or duality relation)
	\begin{align*}
		\E{F\delta(u)}=\E{\inner{DF,u}_\mathfrak{H}}.
	\end{align*}
	\begin{remark}
		A random variable $F$ is an element of $\on{dom}L =
		\mathbb{D}^{2,2}$ if and only if $F \in \on{dom}(\delta D)$
		(that is $F \in \mathbb{D}^{1,2}$ and $DF \in
		\on{dom}(\delta)$), and in this case, 
		\begin{equation*}
			LF = -\delta D F.
		\end{equation*}
	\end{remark}
	\noindent An important result from Malliavin calculus we will make use of in the
	sequel are Meyer's inequalities: for any $1\leq p\leq k$ and $u\in
	\mathbb{D}^{k,q}\left(\mathfrak{H}^{\otimes p}\right)$, there exists a constant
	$C >0$ such that 
	\begin{equation}
		\label{meyerineq}
		\norm{\delta^p(u)}_{k-p,q}\leq
		C\norm{u}_{\mathbb{D}^{k,q}\left( \mathfrak{H}^{\otimes p} \right) }.
	\end{equation}
	For a more complete treatment of Meyer's inequalities, we
        refer to \cite[Theorem 2.5.5]{NP} or \cite[Section 1.5]{Nua}, and related results therein.
	
	\subsection{Distances between random matrices}
	\label{secdistances}
	
	We will use the Wasserstein distance between two random matrices
	taking values in $\mathcal{M}_n(\mathbb{R})$, which denotes the space
	of $n \times n$ real matrices. Given two
	$\mathcal{M}_n(\mathbb{R})$-valued random matrices $\mathcal{X}$ and $\mathcal{Y}$, the Wasserstein
	distance between them is given by 
	\begin{equation*}
		d_W \left(\mathcal{X}, \mathcal{Y}  \right) =
		\sup_{\norm{g}_{\on{Lip}}\leq 1} \abs{\E{ g(\mathcal{X})} -\E{ g(\mathcal{Y}} },
	\end{equation*}
	where the Lipschitz norm $\norm{\cdot}_{\on{Lip}}$ of $g \colon \mathcal{M}_{n}(\mathbb{R}) \to \mathbb{R} $ is defined by  
	
	\begin{equation*}
		\norm{g}_{\on{Lip}} = \sup_{A \neq B\in \mathcal{M}_{n}(\mathbb{R})}\frac{ \abs{ g(A)- g(B)} }{\norm{A-B}_{\on{HS}} },
	\end{equation*}
	with $\norm{\cdot}_{\on{HS}}$ denoting the Hilbert-Schmidt norm on $\mathcal{M}_{n}(\mathbb{R})$.
	\\~\\
	\noindent We will also make use of the Wasserstein distance between random
	vectors, defined analogously as in the matrix case. Namely, if $X, Y$
	are two $n$-dimensional random vectors, then the Wasserstein distance
	between them is defined to be
	\begin{equation*}
		d_W \left(X, Y  \right) =
		\sup_{\norm{g}_{\on{Lip}}\leq 1} \abs{\mathbb{E}\left( g(X)
			\right) -\mathbb{E}\left( g(Y) \right)},
	\end{equation*}
	where the Lipschitz norm $\norm{\cdot}_{\on{Lip}}$ of $g \colon \R^n \to \mathbb{R} $ is defined by  
	
	\begin{equation*}
		\norm{g}_{\on{Lip}} = \sup_{x \neq y \in \R^n}\frac{ \abs{ g(x)- g(y)} }{\norm{x-y}_{\R^n} },
	\end{equation*}
	with $\norm{\cdot}_{\R^n}$ denoting the Euclidean norm on $\R^n$.
	\\~\\
	If $\mathcal{X} = \left( X_{ij} \right)_{1 \leq i,j \leq n}$ is an $n
	\times n$ symmetric random matrix, we associate to it its
	``half-matrix'' defined to be the $n(n+1)/2$-dimensional random vector
	\begin{equation}\label{half}
		\mathcal{X}^{\on{half}} = \left( X_{11}, X_{12} \ldots, X_{1n},
		X_{22}, X_{23}, \ldots , X_{2n}, \ldots, X_{nn}\right).
	\end{equation} 
	It turns out that, in the case of two symmetric matrices, the
	Wasserstein distance between said matrices can be bounded from above
	by a constant multiple of the Wasserstein distance between their
	associated half-matrices. More specifically, we have the following
	Lemma (see \cite[Lemma 2.2]{NZ}).
	\begin{Lemma}\label{ll1}
		Let $\X, \Y$ be two symmetric random matrices with values in $\mathcal{M}_{n}(\mathbb{R}).$ Then
		\begin{equation*}
			d _{W}(\X, \Y) \leq \sqrt{2} d_{W} (\X^{\on{half}}, \Y^{\on{half}}),
		\end{equation*}
		where $\X^{\on{half}}, \Y^{\on{half}}$ are the
		associated half-matrices defined in \eqref{half}.
	\end{Lemma}
	\noindent This shows that assessing the Wasserstein
	distance between symmetric random matrices can be shifted to the
	problem of estimating the Wasserstein distance between associated random
	vectors (see Lemma \ref{ll1}). In our context, a helpful result in
	this direction is \cite[Theorem 6.1.1]{NP}, which we restate here
	for convenience. 
	\begin{prop}[Theorem 6.1.1 in \cite{NP}]
		\label{tt1} 
		Fix $m \geq 2$, and let  $F= (F_{1},\ldots, F_{m}) $ be a centered
		$m$-dimensional random vector with $F_{i}\in \mathbb{D}^{1, 4}$ for
		every $i=1,\ldots,m$. Let $C \in \mathcal{M}_m(\R)$ be a symmetric and
		positive definite matrix, and let $Z\sim N_m(0, C)$. Then,
		\begin{equation*}
			d_{W}(F, Z) \leq  \norm{C^{-1}}_{\on{op}}
			\norm{C}_{\on{op}}^{1/2} \sqrt{ \sum_{i, j=1}^{m} \E{\left(
					C_{ij} - \left\langle DF_{i}, -DL^{-1} F_{j} \right\rangle_{\frak{H}}  \right) ^{2}}},
		\end{equation*}
		where $\norm{\cdot}_{\on{op}}$ denotes the operator norm on $\mathcal{M}_m(\R)$.
	\end{prop}
	
	\section{Proofs of main central convergence results}
	\noindent This section is dedicated to the proofs of the main central
	convergence results, namely Theorems \ref{theoremcentralconvergence}
	and \ref{theoremthreehalf}. Throughout this section and for
	the rest of the paper, $f(x)=o(g(x))$ is taken to mean
	\begin{equation*}
		\lim_{x\to\infty}\frac{f(x)}{g(x)}=0.
	\end{equation*}
	
	\subsection{Proof of Theorem \ref{theoremcentralconvergence}}	
	\label{centralconvergencesection}
	We start by embedding the covariance structure of $\left(
	Y_{ik} \right)_{i,k \in \N}$ in a Hilbert space
	$\mathfrak{H}$ such that for a collection of elements
	$\left(e_{ik}\right)_{i,k\in \N}$ in $\mathfrak{H}$,
	$Y_{ik}=Z(e_{ik})$ and
	$\inner{e_{ik},e_{jl}}_\mathfrak{H}=\E{Y_{ik} Y_{jl}}$. Since
	the rows of $\Y_{n,\floor{dx}}$ are independent and each entry of $\Y_{n,\floor{dx}}$ is
	normalized, we have  $\inner{e_{ik},e_{jl}}_\mathfrak{H}=0$
	for $i\neq j$ and $\norm{e_{ik}}_\mathfrak{H}=1$. For notational convenience, we will denote by
	$\delta_{kj}=\inner{e_{ik},e_{ij}}_\mathfrak{H}$ in the sequel. With
	this structure handy, the
	entries of $\W_{n,\floor{dx}}$ can be represented as
	\begin{equation}
		\label{representationofwij}
		W_{ij}=I_2 \brac{f_{ij}}
	\end{equation}
	for any $1 \leq i,j \leq n$, where the kernel $f_{ij}$ is
	defined by 
	\begin{equation}
		\label{definitionoffij}
		f_{ij} = \frac{1}{2\sqrt{d}}\sum_{k=1}^{\floor{dx}}\brac{e_{ik}\otimes e_{jk}+e_{jk}\otimes e_{ik}}.
	\end{equation}
	Let $\G_n$ denote a $n\times n$ Gaussian matrix, having the same
	covariance structure as $\W_{n,\floor{dx}}$. Using the notation
	introduced in \eqref{half}, we write $\mathcal{W}^{\on{half}}$ to
	denote the half-matrix associated to $\W_{n,\floor{dx}}$, that is
	\begin{align*}
		\mathcal{W}^{\on{half}}=(W_{11},\ldots,W_{1n},W_{22},\ldots,W_{2n},\ldots,W_{nn}),
	\end{align*}
	and we define $\mathcal{G}^{\on{half}}$ in a similar way. As pointed
	out in Lemma \ref{ll1}, we have $d_{W}\left(\W_{n,\floor{dx}},\G_n  \right) \leq
	\sqrt{2}d_{W}\left( \mathcal{W}^{\on{half}},\mathcal{G}^{\on{half}} \right) $ since
	the matrices $\W_{n,\floor{dx}}$ and $\G_n$ are symmetric.
	\begin{remark}
		Slightly
		abusing notation, we will continue to write $\W_{n,\floor{dx}}$ and
		$\G_n$ in place of $\mathcal{W}^{\on{half}}$ and $\mathcal{G}^{\on{half}}$,
		respectively.          
	\end{remark}
	\noindent Our goal being to estimate $d_{W}\left( \W_{n,\floor{dx}}, \Z_n
	\right)$, we apply the triangle inequality to get 
	\begin{align}
		\label{triangleinequality}
		d_{W}\left(\W_{n,\floor{dx}}, \Z_n  \right) \leq
		d_{W}\left(\W_{n,\floor{dx}}, \G_n  \right) +d_{W}\left( \G_n, \Z_n \right) 
	\end{align}
	and split the proof into two steps, the first one aiming at estimating
	$d_{W}\left(\W_{n,\floor{dx}}, \G_n  \right)$, and the second one
	dealing with the estimation of $d_{W}\left( \G_n, \Z_n \right) $.
	\\~\\
	\noindent  {\bf Step 1: Estimation of $d_{W}\left(\W_{n,\floor{dx}}, \G_n  \right)$}.
	\\~\\
	According to Lemma \ref{Lemmacontraction}, we can write
	\begin{align*}
		\norm{f_{ij}\widetilde{\otimes}_{1}f_{lk}}^2_{\mathfrak{H}^{\otimes 2}}\leq  \frac{1}{d^2}\sum_{k,l,m,p=1}^{\floor{dx}}\delta_{kl}\delta_{mp}\delta_{km}\delta_{lp}.
	\end{align*}
	Based on Lemma \ref{covarianceLemmanormal}, it holds that       
	\begin{equation*}
		\norm{\mathcal{C}^{-1}}_{\operatorname{op}}\norm{\mathcal{C}}^{\frac{1}{2}}_{\operatorname{op}}
		= \sqrt{\frac{2d}{\sum_{k,l=1}^{\floor{dx}}{\delta_{kl}}^2}},
	\end{equation*}
	where $\mathcal{C}$ denotes the covariance matrix of $\W^{\operatorname{half}}_{n,\floor{dx}}$. Applying Proposition \ref{tt1} together with Lemmas \ref{secondchaosLemma} and \ref{Lemmacontraction}, and observing that the cardinality of the set 
	\begin{align*}\{(i,j,k,l)\in \left\{ 1,\ldots,n \right\}^4 \colon i,j,k,l \mbox{
			are not mutually distinct}\}\end{align*} is bounded by $6n^3$,
	yields 
	\begin{align}
		\label{estimate2}
		d_{W}\brac{\W_{n,\floor{dx}},\G_n} &\leq \sqrt{2}\norm{\mathcal{C}^{-1}}_{\operatorname{op}}\norm{\mathcal{C}}^{\frac{1}{2}}_{\operatorname{op}}\brac{\sum_{ i,j,k,l=1}^{n} \E{\brac{\E{W_{ij}W_{kl}}-\frac{1}{2}\inner{DW_{ij},DW_{kl}}_{\mathfrak{H}}}^2}}^{\frac{1}{2}}\nonumber\\
		&=\sqrt{2}\brac{\frac{2d}{\sum_{k,l=1}^{\floor{dx}}{\delta_{kl}}^2}}^{\frac{1}{2}}\brac{\sum_{i,j,k,l=1}^{n}8\norm{f_{ij}\widetilde{\otimes_{1}}f_{lk}}^2_{\mathfrak{H}^{\otimes 2}}}^{\frac{1}{2}}\nonumber\\
		&\leq C\brac{\frac{d}{\sum_{k,l=1}^{\floor{dx}}{\delta_{kl}}^2}}^{\frac{1}{2}}\brac{\frac{n^3}{d^2}\sum_{k,l,m,p=1}^{\floor{dx}}\delta_{kl}\delta_{mp}\delta_{km}\delta_{lp}}^{\frac{1}{2}}.
	\end{align}
	We now need to estimate the right hand side of \eqref{estimate2}. We
	divide this estimation into two cases, the first of which deals with
	the case where the process $X$ is such that $\alpha <1$ and $\alpha +
	\nu < 2$, and the second one covering all other possibilities.

	\begin{proofcase}[$\alpha<1$ and $\alpha+\nu<2$]
		Lemma \ref{varianceLemma} implies that the first factor on the right
		hand side of \eqref{estimate2} is bounded, so that it is sufficient
		to estimate the second factor. For $\frac{1}{2}<\alpha<\frac{3}{2}$,
		let $\theta$ be a constant in $(0,1)$ and let
		$M_d=\floor{\brac{dx}^\theta}$. Denote by $D_d$ the set of multi-indexes
		$\left\{1,\ldots,\floor{dx}  \right\}^4$ and decompose $D_d$ into
		$D_{1,M_d}\cup D_{2,M_d}$ according to
		\begin{equation*}
			D_{1,M_d}=\{(k,l,m,p)\in D_d \colon \abs{k-l}\leq M_d,\abs{k-m}\leq M_d, \abs{m-p}\leq M_d\}
		\end{equation*}
		and 
		\begin{equation*}
			D_{2,M_d}=D_{3,M_d}\cup D_{4,M_d} \cup D_{5,M_d},
		\end{equation*}
		where
		\begin{equation*}
			\begin{cases}
				D_{3,M_d}&=\{(k,l,m,p)\in D_d \colon \abs{k-l}> M_d\}\\
				D_{4,M_d}&=\{(k,l,m,p)\in D_d \colon \abs{k-m}> M_d\}\\
				D_{5,M_d}&=\{(k,l,m,p)\in D_d \colon \abs{m-p}> M_d\} 
			\end{cases}.
		\end{equation*}
		Note that the cardinality of $D_{1,M_d}$ is bounded by $8dM_d^3$ since
		there are fewer than $d$ choices for $k$ and for each $k$, one has
		$2M_d$ choices for each $l,m,p$. Hence, we can write
		\begin{align*}
			\frac{n^3}{d^2}\sum_{k,l,m,p=1}^{\floor{dx}}\delta_{kl}\delta_{mp}\delta_{km}\delta_{lp}
			& \leq \frac{n^3}{d^2}\sum_{(k,l,m,p)\in D_{1,M_d}}\delta_{kl}\delta_{mp}\delta_{km}\delta_{lp} + \frac{3n^3}{d^2}\sum_{(k,l,m,p)\in D_{3,M_d}}\delta_{kl}\delta_{mp}\delta_{km}\delta_{lp}\\
			& \leq \frac{8n^3M_d^3}{d}+\frac{3n^3}{d^2}\brac{\sum_{(k,l,m,p)\in D_{3,M_d}}\abs{\delta_{kl}\delta_{mp}}^2}^{\frac{1}{2}} \brac{\sum_{(k,l,m,p)\in D_{3,M_d}}\abs{\delta_{km}\delta_{lp}}^2}^{\frac{1}{2}} \\
			& \leq 8n^3M_d^3 d^{-1}+3n^3\brac{d^{-1}\sum_{\substack{1\leq k,l\leq  \floor{dx}\\\abs{k-l}>M_d}}\abs{\delta_{kl}}^2}^{\frac{1}{2}} \brac{d^{-1}\sum_{m,p=1}^{\floor{dx}}\abs{\delta_{mp}}^2}^{\frac{3}{2}}.
		\end{align*}
		As Lemma \ref{varianceLemma} implies that
		$d^{-1}\sum_{m,p=1}^{\floor{dx}}\abs{\delta_{mp}}^2<\infty$, we get
		\begin{align*}
			\frac{n^3}{d^2}\sum_{k,l,m,p=1}^{\floor{dx}}\delta_{kl}\delta_{mp}\delta_{km}\delta_{lp}
			& \leq
			Cn^3M_d^3 d^{-1}+3Cn^3\brac{d^{-1}\sum_{\substack{1\leq k,l\leq  \floor{dx}\\\abs{k-l}>M_d}}\abs{\delta_{kl}}^2}^{\frac{1}{2}}\\
			& \leq Cn^3 d^{3\theta-1}+3Cn^3\brac{d^{-1}\sum_{\substack{1\leq k,l\leq  \floor{dx}\\\abs{k-l}>M_d}}\abs{\delta_{kl}}^2}^{\frac{1}{2}}.
		\end{align*}
		Now we use Remark \ref{remarkafterlemma5} which states that the dominant part of $d^{-1}\sum_{k, l=1}^{\floor{dx}}\delta_{kl}^2$ is  
		\begin{align*}
			\frac{1}{d}\sum_{m=1}^{\floor{dx}-1}\sum_{k=1}^{\floor{dx}-m-1}\brac{\frac{k}{k+m}}^{2\beta-\alpha}\brac{\abs{m+1}^{\alpha}+\abs{m-1}^{\alpha}-2\abs{m}^{\alpha}}^2+\frac{x}{2},
		\end{align*}
		and the fact that \begin{align*}\abs{m+1}^{\alpha}+\abs{m-1}^{\alpha}-2\abs{m}^{\alpha}=
			\frac{1}{2}\alpha\brac{\alpha-1}\abs{m}^{\alpha-2}+o(\abs{m}^{\alpha-2})\end{align*}
		to get
		\begin{align*}
			n^3\brac{d^{-1}\sum_{\substack{1\leq k,l\leq  \floor{dx}\\
						\abs{k-l}>M_d}}\abs{\delta_{kl}}^2}^{\frac{1}{2}} &\leq Cn^3\brac{\sum_{m=\abs{k-l}=M_d}^ {\floor{dx}-1}\brac{\abs{m+1}^{\alpha}+\abs{m-1}^{\alpha}-2\abs{m}^{\alpha}}^2 }^{\frac{1}{2}} \\
			&\leq Cn^3\brac{\sum_{ m=M_d}
				^{d-1}\alpha^2\brac{\alpha-1}^2m^{2(\alpha-2)}}^{\frac{1}{2}}.
		\end{align*}
		This finally gives us
		\begin{equation*}
			n^3\brac{d^{-1}\sum_{\substack{1\leq k,l\leq  \floor{dx}\\
						\abs{k-l}>M_d}}\abs{\delta_{kl}}^2}^{\frac{1}{2}}
			\leq Cn^3\brac{M_d^{2\alpha-3}}^{\frac{1}{2}} \leq Cn^3 d^{\frac{1}{2}\theta(2\alpha-3)}.
		\end{equation*}
		This gives us the convergence rate
		\begin{align*}
			d_{W}\brac{\W_{n,\floor{dx}},\G_n} \leq  C\brac{n^3 d^{-2}\sum_{k,l,m,p=1}^{\floor{dx}}\delta_{kl}\delta_{mp}\delta_{km}\delta_{lp}}^{\frac{1}{2}}
			&\leq   Cn^\frac{3}{2}\brac{d^{\frac{1}{2}(3\theta-1)}+ d^{\frac{1}{4}\theta(2\alpha-3)}}\\
			&= Cn^\frac{3}{2}d^{\frac{1}{2}(3\theta-1)\vee \frac{1}{4}\theta(2\alpha-3)},
		\end{align*}
		which holds for any $\theta\in (0,\frac{1}{3})$. Now,
		observe that for $\theta\in (0,\frac{1}{3})$ and
		$\alpha\in (0,\frac{3}{2})$, the function
		\begin{align*}
			f(\theta)=\frac{1}{2}(3\theta-1)\vee \frac{1}{4}\theta(2\alpha-3)= 	\begin{cases}
				\displaystyle (1/4)\theta(2\alpha-3) & \mbox{if } \theta \leq \frac{2}{9-2\alpha} \\
				\displaystyle (1/2)(3\theta-1) & \mbox{if } \theta > \frac{2}{9-2\alpha}
			\end{cases}
		\end{align*}
		attains its minimum $\frac{2\alpha-3}{2(9-2\alpha)}$ at $\theta=\frac{2}{9-2\alpha}$. This allows us to conclude that
		\begin{align*}
			d_{W}\brac{\W_{n,\floor{dx}},\G_n}\leq Cn^\frac{3}{2}d^{\frac{2\alpha-3}{2(9-2\alpha)}}.
		\end{align*}
		\begin{remark} The above estimate works for all Wishart matrices such
			that $0<\alpha<\frac{3}{2}$ and not just $\alpha<1$ and $\alpha+\nu<2$,
			but it turns out that we can obtain a better estimate for the other
			cases as the rest of the proof will show.
		\end{remark}
	\end{proofcase}
	\begin{proofcase}[$1 \leq \alpha < \frac{3}{2}$ or $\alpha <1$
		and $\alpha+\nu \geq 2$]
		In this second case, we consider a process $X$ that satisfies any other
		assumptions appearing in Theorem \ref{theoremcentralconvergence}
		besides $\alpha<1$ and $\alpha+\nu<2$, and we obtain the same central
		convergence rate as the one appearing in \cite{NZ} (recall that, as
		mentioned in the introduction, the results in \cite{NZ} are only valid
		under the assumption of stationarity of $X$, which we do not impose
		here).
		\\~\\
		The fact that we get the rate $r(\alpha,\nu)$  in the case
		where $1 \leq \alpha < \frac{3}{2}$ or $\alpha <1$
		and $\alpha+\nu \geq 2$ follows directly from
		Lemma \ref{fbmcomparisonLemma}, the bound \eqref{estimate2} and the
		fact that
		\begin{equation*}
			d^{-1}\brac{\sum_{m=-d+1}^{d-1}\abs{a_\alpha (m)}^{\frac{4}{3}}}^3\leq  C\begin{cases} 
				d^{-1} &\text{if } 0<\alpha<\frac{5}{4} \\
				d^{-1}(\ln{d})^3 &\text{if } \alpha=\frac{5}{4}\\
				d^{4\alpha-6} &\text{if } \frac{5}{4}<\alpha<2
			\end{cases}.
		\end{equation*}
	\end{proofcase}
	~\\~\\
	\noindent  {\bf Step 2: Estimation of $d_{W}\left(\G_n, \Z_n
		\right)$}.
	\\~\\	
	This step in concerned with bounding the Wasserstein distance between
	$\G_n$ and $\Z_n$. Applying Proposition \ref{tt1} yields
	\begin{align*}
		d_{W}\brac{\G_n,\Z_n}\leq \sqrt{2}\norm{\mathcal{C_\Z}^{-1}}_{\operatorname{op}}\norm{\mathcal{C_\Z}}^{\frac{1}{2}}_{\operatorname{op}}\brac{\sum_{\substack{1\leq i,j,l,k\leq n\\i\leq j; l\leq k}}\E{\brac{\E{Z_{ij}Z_{lk}}-\inner{DG_{ij},DG_{lk}}_\mathfrak{H}}^2}}^{\frac{1}{2}},
	\end{align*}
	where $C_\Z$ denotes the covariance matrix of $\Z_n$. Recall
	that $Z_{ii}\sim N(0,2\sigma^2)$, that $Z_{ij}\sim
	N(0,\sigma^2)$ for $i\neq j$ and that all the entries of
	$\Z_n$ are independent. Hence, Lemma $\ref{varianceLemma}$
	provides us with the exact values of $\E{Z_{ij}Z_{lk}}$. 
	Lemma $\ref{varianceLemma}$ also implies that $\sigma^2 \leq
	\norm{\mathcal{C_\Z}}_{\operatorname{op}}\leq 2\sigma^2$. Meanwhile,
	$G_{ij}$ and $G_{lk}$ are in the first Wiener chaos associated
	to $X$, so that
	$\inner{DG_{ij},DG_{lk}}_\mathfrak{H}=\E{G_{ij}G_{lk}}$. Thus,
	Lemma \ref{covarianceLemmanormal}, Remark
	\ref{remarkafterlemma5} and the fact that $\G_n$ and
	$\W_{n,\floor{dx}}$ are identically distributed yield
	\begin{align*}
		\E{G_{ii}G_{ii}} =2d^{-1}\sum_{k,j=1}^{\floor{dx}}\delta_{kj}^2
		=
		\frac{1}{d}\sum_{m=1}^{\floor{dx}-1}\sum_{k=1}^{\floor{dx}-m-1}\brac{\frac{k}{k+m}}^{2\beta-\alpha}a^2_\alpha(m)+\frac{x}{2}a^2_\alpha(0)+R_d,
	\end{align*}
	and  for $i\neq l$,
	\begin{align*}       
		\E{G_{il}G_{il}}=d^{-1}\sum_{k,j=1}^{\floor{dx}}\delta_{kj}^2
		= \frac{1}{2d}\sum_{m=1}^{\floor{dx}-1}\sum_{k=1}^{\floor{dx}-m-1}\brac{\frac{k}{k+m}}^{2\beta-\alpha}a^2_\alpha(m)+\frac{x}{4}a^2_\alpha(0)+R_d,
	\end{align*}
	where $R_d=o\brac{d^{2\alpha-3}+d^{-1}}$. 
	\\~\\	
	Recalling that $a^2_\alpha(m)=\frac{1}{4}\brac{\abs{m+1}^{\alpha}+\abs{m-1}^{\alpha}-2\abs{m}^\alpha}^2$,
	we are able to get
	\begin{align}
		\label{estimateGandZ}
		d_{W}\brac{\G_n,\Z_n}& \leq \sqrt{2}\frac{\sqrt{2\sigma^2}}{\sigma^2}\brac{\frac{n(n+1)}{2}\brac{2\sigma^2 -2d^{-1}\sum_{k,j=1}^{\floor{dx}}{\delta_{k{j}}}^2}^{2}}^{\frac{1}{2}}\\
		&\leq C\frac{\sqrt{n(n+1)}}{\sigma}(A_1+A_2-R_d),\nonumber
	\end{align}
	where
	\begin{equation*}
		A_1=\frac{x}{2}\sum_{m\in\mathbb{Z}}a^2_\alpha(m)-\frac{1}{d}\sum_{m=1}^{\floor{dx}-1}\brac{\floor{dx}-m-1}a^2_\alpha(m)-\frac{x}{2}a^2_\alpha(0)
	\end{equation*}
	and
	\begin{equation*}
		A_2=\frac{1}{d}\sum_{m=1}^{\floor{dx}-1}\brac{\floor{dx}-m-1}a^2_\alpha(m)-\frac{1}{d}\sum_{m=1}^{\floor{dx}-1}\sum_{k=1}^{\floor{dx}-m-1}\brac{\frac{k}{k+m}}^{2\beta-\alpha}a^2_\alpha(m).
	\end{equation*}
	As $d$ gets sufficiently large, the term $A_1$ can be bounded by
	\begin{align*}
		A_1&=\sum_{m=1}^{\floor{dx}-1}\frac{m+1}{d}a^2_\alpha(m)+ x\sum_{m=\floor{dx}}^{\infty}a^2_\alpha(m)<Cd^{2\alpha-3}
	\end{align*}
	for $\alpha<\frac{3}{2}$ and $\alpha\neq 1$. Meanwhile, to deal with the term
	$A_2$, we observe that as $x\to\infty$, $\ln{\abs{x}}\leq
	\abs{x}^\zeta$ for any positive value of $\zeta$. We also know
	that $a_\alpha(m)=\frac{1}{2}\alpha(\alpha-1)m^{\alpha-2}+o\brac{m^{\alpha-2}}$. Thus, a Taylor expansion of $(1-x)^p$ for $p>0$ and $0\leq x<1$ gives us
	\begin{align*}
		A_2&= \frac{C}{d}\sum_{m=1}^{\floor{dx}-1}\sum_{k=1}^{\floor{dx}-m-1}\brac{1-\brac{1-\frac{m}{k+m}}^{2\beta-\alpha}}m^{2(\alpha-2)}\\
		&\leq \frac{C}{d}\sum_{m=1}^{\floor{dx}-1}\sum_{k=1}^{\floor{dx}-m-1}\frac{m}{k+m}m^{2(\alpha-2)}\\
		&\leq \frac{C}{d}\sum_{m=1}^{\floor{dx}-1}m^{2\alpha-3}\int_1^{\floor{dx}-m-1}\frac{1}{y+m-1}dy
		\\ &\leq \frac{C}{d}\sum_{m=1}^{\floor{dx}-1}m^{2\alpha-3} \ln{\frac{\floor{dx}}{m}}\\
		&\leq \frac{C}{d}\sum_{m=1}^{\floor{dx}-1}m^{2\alpha-3} \brac{\frac{\floor{dx}}{m}}^\zeta \leq Cd^{\zeta-1}\sum_{m=1}^{\floor{dx}-1}m^{2\alpha-\zeta-3} \leq Cd^{2\alpha-3}.
	\end{align*}
	Earlier we have mentioned that $R_d=o\brac{d^{2\alpha-3}+d^{-1}}$. This allows us to conclude that, for $\alpha<\frac{3}{2}$ and
	$\alpha\neq 1$, we have
	\begin{align*}
		d_{W}\brac{\G_n,\Z_n} \leq C\frac{\sqrt{n(n+1)}}{\sigma}\brac{d^{2\alpha -3}+d^{-1}}
	\end{align*}
	for any $\zeta\in(0,1)$. As for $\alpha=1$, $a_{1}(m) = 0$ if $m\neq0$ and $a_1(0)=1$, so the estimate \eqref{estimateGandZ} becomes
	\begin{align*}
		d_{W}\brac{\G_n,\Z_n}\leq C\brac{x-\frac{\floor{dx}}{d}-\frac{1}{d}}\leq \frac{C}{d}.
	\end{align*}
	Finally the estimate in Theorem \ref{theoremcentralconvergence} follows
	immediately from \eqref{triangleinequality}. To conclude that
	$\W_{n,\floor{dx}}$ is close to $\Z_n$ in the sense of
	finite-dimensional distributions, we refer to \cite[Theorem 6.2.3]{NP}
	which states that for a sequence of vectors of multiple Wiener
	integrals, component-wise convergence to a Gaussian limit implies joint convergence.    
	\subsection{Proof of Theorem \ref{theoremthreehalf}}
	\label{threehalfsection}
	\noindent In the case $\alpha=\frac{3}{2}$, recall that we have to use
	a different normalization for the elements of the Wishart matrix (we
	hence adjust notation accordingly), namely for $i\leq j$,
	\begin{equation*}
		\widetilde{W}_{ij}\brac{\floor{dx}}=\frac{1}{\sqrt{d\ln d}}\sum_{k=1}^{\floor{dx}}\left(Y^i_kY^j_k -\mathds{1}_{\left\{ i=j \right\}}\right).
	\end{equation*}
	Let $\widetilde{\G}_n$ be $n \times n$ Gaussian matrices with the same covariance
	structure as $\widetilde{\W}_{n,\floor{dx}}$, which as before denotes a half-matrix vector. In the same spirit as in
	the proof of Theorem \ref{theoremcentralconvergence}, we will first
	estimate the Wasserstein distance from $\widetilde{\W}_{n,\floor{dx}}$
	to $\widetilde{\G}_n$ in Step 1 and the distance from
	$\widetilde{\G}_n$ to the G.O.E matrix $\widetilde{\Z}_n$ in Step
	2. Theorem \ref{theoremthreehalf} then follows from the triangle inequality
	\begin{align}
		\label{triangleinequalitythreehalf}
		d_{W}(\widetilde{\W}_{n,\floor{dx}},\widetilde{\Z}_n)\leq d_{W}(\widetilde{\W}_{n,\floor{dx}},\widetilde{\G}_n)+d_{W}(\widetilde{\G}_{n},\widetilde{\Z}_n).
	\end{align}
	~\\
	\noindent  {\bf Step 1: Estimation of $d_{W}(\widetilde{\W}_{n,\floor{dx}},\widetilde{\G}_n)$}.
	\\~\\
	We proceed in the same way as in the beginning of the proof of Theorem
	\ref{theoremcentralconvergence}. Denote by $\widetilde{\mathcal{C}}$
	the covariance matrix of $\widetilde{\W}_{n,\floor{dx}}$ and
	$\widetilde{\G}_n$. Use Proposition \ref{tt1} and observe that Lemmas
	\ref{secondchaosLemma} and \ref{Lemmacontraction} still hold for
	$\alpha=\frac{3}{2}$, modulo a change of normalizing factor. More
	precisely, we can write
	\begin{align*}
		d_{W}(\widetilde{\W}_{n,\floor{dx}},\widetilde{\G}_n) &\leq \sqrt{2}\norm{\widetilde{\mathcal{C}}^{-1}}_{\operatorname{op}}\norm{\widetilde{\mathcal{C}}}^{\frac{1}{2}}_{\operatorname{op}}\brac{\sum_{i,j,l,k=1}^n\E{\brac{\E{\widetilde{W}_{ij}\widetilde{W}_{lk}}-\frac{1}{2}\inner{D\widetilde{W}_{ij},D\widetilde{W}_{lk}}_{_\mathfrak{H}}}^2}}^{\frac{1}{2}}\\
		&\leq C\brac{\frac{d\ln{d}}{\sum_{k,l=1}^{\floor{dx}}{\delta_{kl}}^2}}^{\frac{1}{2}}\brac{\frac{n^3}{(d\ln{d})^2}\sum_{k,l,m,p=1}^{\floor{dx}}\delta_{kl}\delta_{mp}\delta_{km}\delta_{lp}}^{\frac{1}{2}},
	\end{align*}
	at which point we can use Lemmas  \ref{fbmcomparisonLemmathreehalf}
	and \ref{varianceLemmathreehalf} to get
	\begin{align*}
		d_{W}(\widetilde{\W}_{n,\floor{dx}},\widetilde{\G}_n)&\leq  C \brac{ \frac{n^3}{(d\ln{d})^2}\sum_{k,l,m,p=1}^{\floor{dx}}\delta_{kl}\delta_{mp}\delta_{km}\delta_{lp}}^{\frac{1}{2}}\\
		&\leq  C \brac{\frac{n^3}{d(\ln{d})^2}
			\brac{\sum_{m=-\floor{dx}+1}^{\floor{dx}}\abs{a_\frac{3}{2}
					(m)}^{\frac{4}{3}}}^3}^{\frac{1}{2}}
	\end{align*}
	We have $a_{\frac{3}{2}}
	(m)=\frac{1}{2}\brac{\abs{m+1}^{\frac{3}{2}}+\abs{m-1}^{\frac{3}{2}}-2\abs{m}^{\frac{3}{2}}}=\frac{3}{8}
	\abs{m}^{-\frac{1}{2}}+o(\abs{m}^{-\frac{1}{2}})$. Combining with the fact
	that 
	\begin{equation*}
		\sum_{m=1}^{\floor{dx}}m^{-2/3}\leq \int_{1}^{\floor{dx}}(y-1)^{-2/3}dy\leq d^{1/3},
	\end{equation*}
	we obtain
	\begin{equation*}
		d_{W}(\widetilde{\W}_{n,\floor{dx}},\widetilde{\G}_n)\leq  C \brac{\frac{n^3}{d(\ln{d})^2}\brac{\sum_{\substack{m=-\floor{dx}+1\\m\neq 0}}^{\floor{dx}}\abs{m}^{-2/3}+1}^3}^{\frac{1}{2}}
		\leq  C \frac{n^{3/2}}{\ln{d}}.
	\end{equation*}
	~\\
	\noindent  {\bf Step 2: Estimation of
		$d_{W}(\widetilde{\G}_{n},\widetilde{\Z}_n)$}.
	Denote by $\mathcal{C}_{\widetilde{\Z}}$ the covariance matrix of
	$\widetilde{\Z}_n$. Proposition \ref{tt1} implies that
	\begin{align*}
		d_{W}(\widetilde{\G}_n,\widetilde{\Z}_n)&\leq \norm{\mathcal{C}_{\widetilde{\Z}}^{-1}}_{\operatorname{op}}\norm{\mathcal{C}_{\widetilde{\Z}}}^{\frac{1}{2}}_{\operatorname{op}}\brac{\sum_{\substack{1\leq i,j,l,k\leq n\\i\leq j;\ l\leq k}}\E{\brac{\E{\widetilde{Z}_{ij}\widetilde{Z}_{lk}}-\inner{D\widetilde{G}_{ij},D\widetilde{G}_{lk}}_\mathfrak{H}}^2}}^{\frac{1}{2}}.
	\end{align*}
	To estimate the above expression, we notice that since $\widetilde{Z}_{ii}\sim N(0,2\rho^2)$, $\widetilde{Z}_{ij}\sim N(0,\rho^2)$ for $i\neq j$ and all of the entries of $\widetilde{\Z}_n$ are independent, Lemma $\ref{varianceLemmathreehalf}$ provides us with the exact value of $\E{Z_{ij}Z_{lk}}$. Lemma $\ref{varianceLemmathreehalf}$ also implies that $\rho^2 \leq
	\norm{\mathcal{C}_{\widetilde{Z}}}_{\operatorname{op}}\leq 2\rho^2$. Meanwhile, $\widetilde{G}_{ij}$ and $\widetilde{G}_{lk}$ are in the first Wiener chaos associated to $X$, so that $\inner{D\widetilde{G}_{ij},D\widetilde{G}_{lk}}_\mathfrak{H}=\E{\widetilde{G}_{ij}\widetilde{G}_{lk}}$. Combined with the fact that $\widetilde{\G}_n$ and $\widetilde{\W}_{n,\floor{dx}}$ are identically distributed, again an application of Lemma \ref{varianceLemmathreehalf} yields
	\begin{equation*}
		\E{\widetilde{G}_{ii}\widetilde{G}_{ii}}=2\frac{1}{d\ln
			d}\sum_{k,j=1}^{\floor{dx}}\delta_{kj}^2 =
		\frac{1}{d\ln
			d}\sum_{m=1}^{\floor{dx}-1}\sum_{k=1}^{\floor{dx}-m-1}\brac{\frac{k}{k+m}}^{2\beta-\frac{3}{2}}a_\frac{3}{2}(m)^2+\widetilde{R}_d
	\end{equation*}
	and for $i\neq l$,
	\begin{equation*}
		\E{\widetilde{G}_{il}\widetilde{G}_{il}}=\frac{1}{d\ln d}\sum_{k,j=1}^{\floor{dx}}\delta_{kj}^2
		= \frac{1}{2d\ln d}\sum_{m=1}^{\floor{dx}-1}\sum_{k=1}^{\floor{dx}-m-1}\brac{\frac{k}{k+m}}^{2\beta-\frac{3}{2}}a_\frac{3}{2}(m)^2+\widetilde{R}_d,
	\end{equation*}
	with $\widetilde{R}_d=o(\frac{1}{\ln d})$, which can be deduced from the proof of Lemma \ref{varianceLemmathreehalf} (similarly to what was done in Step $2$ in the proof of Theorem \ref{theoremcentralconvergence}). Now keep in mind that $a_\frac{3}{2}(m)$ is defined in Lemma \ref{fbmcomparisonLemmathreehalf} and $a_\frac{3}{2}(m)^2= \frac{9}{64}m^{-1}+o(m^{-1})$. This yields
	\begin{align}
		\label{estimateGandZthreehalf}
		d_{W}\brac{\widetilde{\G}_n,\widetilde{\Z}_n}& \leq \sqrt{2}\frac{\sqrt{2\rho^2}}{\rho^2}\brac{\frac{n(n+1)}{2}\brac{2\rho^2 -2d^{-1}\sum_{k,j=1}^{\floor{dx}}{\delta_{k{j}}}^2}^{2}}^{\frac{1}{2}}\\
		&\leq C\frac{\sqrt{n(n+1)}}{\rho}\brac{A_3+A_4-\widetilde{R}_d}\nonumber,
	\end{align}
	where
	\begin{equation*}
		A_3=x-\frac{1}{d\ln
			d}\sum_{m=1}^{\floor{dx}-1}\brac{\floor{dx}-m-1}m^{-1}
	\end{equation*}
	and
	\begin{equation*}          
		A_4=\frac{1}{d\ln d}\sum_{m=1}^{\floor{dx}-1}\brac{\floor{dx}-m-1}m^{-1}-\frac{1}{d\ln d}\sum_{m=1}^{\floor{dx}-1}\sum_{k=1}^{\floor{dx}-m-1}\brac{\frac{k}{k+m}}^{2\beta-\alpha}m^{-1}.
	\end{equation*}
	For the term $A_3$, asymptotically, we can write
	\begin{align*}
		A_3&\leq C\brac{x-\frac{1}{d\ln d}\int_{1}^{\floor{dx}-1}\frac{\floor{dx}-y-1}{y}dy}\\
		&=C\brac{x-\frac{1}{d\ln d}\int_{1}^{\floor{dx}-1}\left(\frac{\floor{dx}}{y}-1-\frac{1}{y}\right)dy}\\
		&=C\brac{x-\frac{\floor{dx}\ln{\brac{\floor{dx}-1}}}{d\ln
				d}+\frac{1}{d\ln d}+\frac{\ln{\brac{\floor{dx}-1}}}{d\ln
				d}}\\
		&=C\left(x\brac{1-\frac{\ln{\brac{\floor{dx}-1}}}{\ln d}}-\brac{x-\frac{\floor{dx}}{d}}\frac{\ln{\brac{\floor{dx}-1}}}{\ln d} +\frac{1}{d\ln
			d}+\frac{\ln{\brac{\floor{dx}-1}}}{d\ln d}\right)\\
		& \leq \frac{C}{\ln d}.
	\end{align*}
	For the term $A_4$, observe that as $x\to\infty$, $\ln{\abs{x}}\leq
	\abs{x}^\zeta$ for any positive value of $\zeta$. Thus, a Taylor expansion of $(1-x)^p$ for $p>0$ and $0\leq x<1$ yields
	\begin{align*}
		A_4&= \frac{C}{d\ln d}\sum_{m=1}^{\floor{dx}-1}\sum_{k=1}^{\floor{dx}-m-1}\brac{1-\brac{1-\frac{m}{k+m}}^{2\beta-\alpha}}m^{-1}\\
		&\leq \frac{C}{d\ln d}\sum_{m=1}^{\floor{dx}-1}\sum_{k=1}^{\floor{dx}-m-1}\frac{m}{k+m}m^{-1}\\
		&\leq \frac{C}{d\ln d}\sum_{m=1}^{\floor{dx}-1}\int_1^{\floor{dx}-m-1}\frac{1}{y+m-1}dy
		\\ &= \frac{C}{d\ln d}\sum_{m=1}^{\floor{dx}-1} \ln{\frac{\floor{dx}}{m}}\\
		&\leq \frac{C}{d\ln d}\sum_{m=1}^{\floor{dx}-1} \brac{\frac{\floor{dx}}{m}}^\zeta \leq \frac{Cd^\zeta}{d\ln d}\sum_{m=1}^{\floor{dx}-1}\brac{\frac{1}{m}}^\zeta \leq \frac{C}{\ln d}.
	\end{align*}
	Since $\widetilde{R}_d=o(\frac{1}{\ln d})$, the estimate \eqref{estimateGandZthreehalf} then becomes
	\begin{equation*}
		d_{W}\brac{\widetilde{\G}_n,\widetilde{\Z}_n} \leq C\frac{\sqrt{n(n+1)}}{\ln d},
	\end{equation*}
	and the estimate in Theorem \ref{theoremthreehalf} follows
	from \eqref{triangleinequalitythreehalf}. Like for the proof of Theorem \ref{theoremcentralconvergence}, the conclusion that $\widetilde{\W}_{n,\floor{dx}}$ is close to the
	G.O.E matrix $\widetilde{\Z}_n$ in finite-dimensional distributions
	follows once again from \cite[Theorem 6.2.3]{NP}.

	\section{Proof of the main non-central convergence result}
	\noindent As pointed out in Theorem \ref{theoremrosenblatt}, the case where $\frac{3}{2}\leq
	\alpha<2$ unveils a interesting phenomenon of non-central
	convergence, giving rise to a limiting object known as the Rosenblatt-Wishart
	matrix, introduced in \cite{NZ}. This section is dedicated to the
	proof of this non-central convergence result.
	\label{rosenblattconvergencesection}
	\subsection{Proof of Theorem \ref{theoremrosenblatt}}
	\label{proofoftheorosenblatt}
	The self-similarity property of $X$ implies that the entries 
	\begin{equation*}
		\widehat{Y}^i_{k}=\frac{X^i_{\frac{k+1}{d}}-X^i_{\frac{k}{d}}}{\norm{X^i_{\frac{k+1}{d}}-X^i_{\frac{k}{d}}}_{L^2{(\Omega)}}}
	\end{equation*}
	of the $n \times d$ matrix $\widehat{\Y}$ are equal in distribution to
	the entries $Y_k^i$ of $\Y$. Given that the statement we aim to prove
	is distributional, we can work with (keeping the same denomination by a
	slight abuse of notation) the matrix 
	\begin{equation*}
		\widehat{\W}_{n,\floor{dx}}=\frac{\floor{dx}}{d^{\alpha-1}}\brac{\frac{1}{\floor{dx}}\Y\Y^T-I}
	\end{equation*}
	with entries
	\begin{equation*}
		\widehat{W}_{ij}\brac{\floor{dx}}=d^{1-\alpha}\sum_{k=1}^{\floor{dx}}\left(\widehat{Y}^i_k\widehat{Y}^j_k -\mathds{1}_{\left\{ i=j \right\}}\right)
	\end{equation*}
	in place of the original one given by \eqref{defrosmat}. The existence
	of the limit (in the $L^2 \left( \Omega \right)$ sense) of $\widehat{\W}_{n,\floor{dx}}$ as $d$ goes to infinity, called the
	Rosenblatt-Wishart matrix $\Ro_n$, is ensured by Lemma
	\ref{Lemmal2convergence}. To estimate the Wasserstein distance between our Wishart
	matrix and the Rosenblatt-Wishart matrix, we need the following result from
	\cite{Br} applied to the half-matrices associated with
	$\widehat{\W}_{n,\floor{dx}}$ and $\Ro_n$, which in our context, reads
	\begin{align}
		\label{estimateWassdistboundbyL2norm}
		d_{W}\brac{\widehat{W}_{n,\floor{dx}},\Ro_n}&\leq \sqrt{2\sum_{1\leq i\leq j\leq n}\E{(\widehat{W}_{ij}\brac{\floor{dx}}-R_{ij})^2}}.
	\end{align}
	We hence need to evaluate, for all $1 \leq i \leq l \leq n$,
	\begin{align}
		\label{expandsquare}
		\E{\brac{\widehat{W}_{il}\brac{\floor{dx}}-R_{il}}^2}= \E{\widehat{W}_{il}\brac{\floor{dx}}^2}-2\E{\widehat{W}_{il}\brac{\floor{dx}}R_{il}}+\E{R_{il}^2}.
	\end{align}
	\cite[Lemma 3.2]{HN} states that
	$\E{\left(X^i_{\frac{k+1}{d}}-X^i_{\frac{k}{d}}\right)^2}=2\lambda \brac{\frac{k}{d}}^{2\beta-\alpha}\brac{\frac{1}{d}}^{\alpha}(1+\eta_{k,d})$
	where $\eta_{k,d}=o(k^{\alpha-2})$. This fact combined with Lemma
	\ref{covarianceLemmarosenblatt} allows us to estimate the first term
	on the right hand side of \eqref{expandsquare} as
	\begin{align*}
		\E{\widehat{W}_{il}\brac{\floor{dx}}^2} &= \frac{1}{4\lambda^2}\sum_{1\leq k,j\leq  \floor{dx}}d^{2(1-\alpha)}\brac{\brac{\frac{k}{d}}^{2\beta-\alpha}\brac{\frac{1}{d}}^{\alpha}(1+\eta_{k,d})}^{-1}\\
		&\qquad\qquad \brac{\brac{\frac{j}{d}}^{2\beta-\alpha}\brac{\frac{1}{d}}^{\alpha}(1+\eta_{j,d})}^{-1}\brac{\int_{\frac{j}{d}}^{\frac{j+1}{d}}\int_{\frac{k}{d}}^{\frac{k+1}{d}}\partial_{s,t}\E{X_s X_t}dsdt}^2\\
		&= \frac{1}{4\lambda^2}\sum_{1\leq k,j\leq
			\floor{dx}}d^{2-2(\alpha-2\beta)}\bigg((kj)^{\alpha-2\beta}+o\brac{(kj)^{\alpha-2\beta}k^{\alpha-2}}\\
		&\quad +o\brac{(kj)^{\alpha-2\beta}j^{\alpha-2}}\bigg)\brac{\int_{\frac{j}{d}}^{\frac{j+1}{d}}\int_{\frac{k}{d}}^{\frac{k+1}{d}}\partial_{s,t}\E{X_s X_t}dsdt}^2,
	\end{align*}
	for which we have used
	$\brac{1+\eta_{k,d}}^{-1}=1+o(k^{\alpha-2})$. Next, we use the
	substitution
	$s=\frac{u+k}{d}$, $t = \frac{v+j}{d}$ and apply the mean value theorem to get
	\begin{align*}
		\E{\widehat{W}_{il}\brac{\floor{dx}}^2} =\frac{d^{2-2(\alpha-2\beta)}}{4\lambda^2}\sum_{1\leq k,j\leq  \floor{dx}} \bigg((kj)^{\alpha-2\beta}+o\brac{(kj)^{\alpha-2\beta}k^{\alpha-2}}&\\+o\brac{(kj)^{\alpha-2\beta}j^{\alpha-2}}\bigg)\evalat{\partial_{u,v}\E{X_{\frac{u+k}{d}} X_{\frac{v+j}{d}}}^2}_{\substack{u=u_0\\v=v_0}},
	\end{align*}
	with $u_0,v_0\in [a,b]$. For the second term on the right hand side of
	\eqref{expandsquare}, an application of Lemmas
	\ref{covarianceLemmarosenblatt} and \ref{mvtconsequence} yields
	\begin{align*}
		\E{\widehat{W}_{il}\brac{\floor{dx}}R_{il}} &= \lim_{p\to \infty}\E{\widehat{W}_{il}\brac{\floor{dx}}\widehat{W}_{il}(p)}\\
		&= \frac{1}{4\lambda^2}\sum_{1\leq k\leq  \floor{dx}}d^{1-\alpha}\int_0^1\brac{\brac{\frac{k}{d}}^{2\beta-\alpha}\brac{\frac{1}{d}}^{\alpha}(1+\eta_{k,d})}^{-1}\\
		&\qquad\qquad \qquad\qquad \qquad\qquad \qquad\qquad t^{\alpha-2\beta}\brac{\int_{\frac{k}{d}}^{\frac{k+1}{d}}\partial_{s,t}\E{X_s X_t}ds}^2dt\\
		&= \frac{1}{4\lambda^2}\sum_{1\leq k,j\leq  \floor{dx}}d^{1-\alpha}\int_{\frac{j}{d}}^{\frac{j+1}{d}}\brac{\brac{\frac{k}{d}}^{2\beta-\alpha}\brac{\frac{1}{d}}^{\alpha}(1+\eta_{k,d})}^{-1}\\
		&\qquad\qquad \qquad\qquad \qquad\qquad \qquad\qquad t^{\alpha-2\beta}\brac{\int_{\frac{k}{d}}^{\frac{k+1}{d}}\partial_{s,t}\E{X_s X_t}ds}^2dt\\
		&=\frac{d^{2-2(\alpha-2\beta)}}{4\lambda^2}\sum_{1\leq k,j\leq
			\floor{dx}}\int_0^1
		\brac{k^{\alpha-2\beta}+k^{(\alpha-2\beta)+(\alpha-2)}}\\
		&\qquad
		\qquad\qquad
		\qquad\qquad (v+j)^{\alpha-2\beta}\brac{\int_{0}^{1}\partial_{u,v}\E{X_{\frac{u+k}{d}} X_{\frac{v+j}{d}}}du}^2dv\\
		&=\frac{d^{2-2(\alpha-2\beta)}}{4\lambda^2}\sum_{1\leq k,j\leq
			\floor{dx}}
		\brac{k^{\alpha-2\beta}+k^{(\alpha-2\beta)+(\alpha-2)}}\\
		&\qquad\qquad
		\qquad\qquad
		\qquad\qquad (v_1+j)^{\alpha-2\beta}\evalat{\partial_{u,v}\E{X_{\frac{u+k}{d}} X_{\frac{v+j}{d}}}^2}_{\substack{u=u_1\\v=v_1}},
	\end{align*}
	where $u_1, v_1\in [a,b]$. To compute the last term of the right hand
	side of \eqref{expandsquare}, we use Lemma
	\ref{covarianceLemmarosenblatt} once more, which allows us to write
	\begin{align*}
		\E{R_{ij}^2} &= \lim_{d\to \infty}\E{\widehat{W}_{il}\brac{\floor{dx}}^2}\\
		&=\frac{1}{4\lambda^2}\int_0^1 \int_0^1 (st)^{\alpha-2\beta}\brac{\partial_{s,t}\E{X_s X_t}}^2dsdt\\
		&=\frac{1}{4\lambda^2}\sum_{1\leq k,j\leq  \floor{dx}}\int_{\frac{k}{d}}^{\frac{k+1}{d}} \int_{\frac{j}{d}}^{\frac{j+1}{d}} (st)^{\alpha-2\beta}\brac{\partial_{s,t}\E{X_s X_t}}^2dsdt\\
		&=\frac{d^{2-2(\alpha-2\beta)}}{4\lambda^2}\sum_{1\leq k,j\leq  \floor{dx}}\int_0^1 \int_0^1 (u+k)^{\alpha-2\beta}(v+j)^{\alpha-2\beta}\brac{\partial_{u,v}\E{X_{\frac{u+k}{d}} X_{\frac{v+j}{d}}}}^2dudv\\
		&=\frac{d^{2-2(\alpha-2\beta)}}{4\lambda^2}\sum_{1\leq k,j\leq  \floor{dx}} (u_2+k)^{\alpha-2\beta}(v_2+j)^{\alpha-2\beta}\evalat{\brac{\partial_{u,v}\E{X_{\frac{u+k}{d}} X_{\frac{v+j}{d}}}}^2}_{\substack{u=u_2\\v=v_2}},
	\end{align*}
	where $u_2, v_2\in [a,b]$. Therefore, combining these estimates, we get
	\begin{align*}
		\E{\brac{\widehat{W}_{il}\brac{\floor{dx}}-R_{il}}^2} &\leq E_1+E_2,
	\end{align*}
	where
	\begin{align*}
		E_1&=C d^{2-2(\alpha-2\beta)}\sum_{1\leq k,j\leq  \floor{dx}}
		\sup_{u_3,v_3\in [a,b]}\evalat{\partial_{u,v}\E{X_{\frac{u+k}{d}} X_{\frac{v+j}{d}}}^2}_{\substack{u=u_3\\v=v_3}}\\
		&\qquad
		\qquad\qquad
		\qquad\qquad \brac{(kj)^{\alpha-2\beta}
			+(u_2+k)^{\alpha-2\beta}(v_2+j)^{\alpha-2\beta}-2k^{\alpha-2\beta}(v_1+j)^{\alpha-2\beta}}
	\end{align*}
	and
	\begin{align*}
		E_2&=C d^{2-2(\alpha-2\beta)}\sum_{1\leq k,j\leq
			\floor{dx}}
		(kj)^{\alpha-2\beta}\bigg(\evalat{\partial_{u,v}\E{X_{\frac{u+k}{d}}
				X_{\frac{v+j}{d}}}^2}_{\substack{u=u_0\\v=v_0}}\\ &
		\qquad
		\qquad\qquad
		\qquad\qquad+\evalat{\partial_{u,v}\E{X_{\frac{u+k}{d}} X_{\frac{v+j}{d}}}^2}_{\substack{u=u_2\\v=v_2}}-2\evalat{\partial_{u,v}\E{X_{\frac{u+k}{d}} X_{\frac{v+j}{d}}}^2}_{\substack{u=u_1\\v=v_1}}\bigg).
	\end{align*}
	For $E_1$ we can use the estimate on $\partial_{s,t}\E{X_s X_t}$ in
	Lemma \ref{covarianceLemmarosenblatt} and the symmetry of $j\leq k$
	and $j\geq k$ to get 
	\begin{align*}
		E_1& \leq C d^{2-2(\alpha-2\beta)}\sum_{\substack{1\leq j\leq k\leq  \floor{dx}\\\frac{j}{k}\leq \frac{1}{2}}}
		j^{2(2\beta-\alpha)}k^{2(\alpha-2)}d^{-4\beta}\brac{k^{\alpha-2\beta}j^{\alpha-2\beta-1}}\\
		&\quad +C d^{2-2(\alpha-2\beta)}\sum_{\substack{1\leq j< k\leq  \floor{dx}\\\frac{1}{2}<\frac{j}{k}}}
		j^{2(2\beta-\alpha)}\abs{k-j}^{2(\alpha-2)}d^{-4\beta}\brac{k^{\alpha-2\beta}j^{\alpha-2\beta-1}}\\
		&\leq C d^{2-2\alpha}\sum_{k=1}^{\floor{dx}}\sum_{j=1}^{\floor{\frac{k}{2}}}
		k^{2(\alpha-2)}j^{-1}+C d^{2-2\alpha}\sum_{j=1}^{\floor{dx}-1}\sum_{k=j+1}^{2j}\abs{k-j}^{2(\alpha-2)}j^{-1}\\
		&\leq C d^{2-2\alpha}\sum_{k=1}^{\floor{dx}}k^{2(\alpha-2)}\ln k
		+d^{2-2\alpha}\sum_{j=1}^{\floor{dx}-1}j^{2(\alpha-2)}.
	\end{align*}
	As $k\to\infty$, it holds that $\ln k\leq k^\theta$ for any $\theta \in (0,1)$,  so that
	\begin{equation*}
		E_1 \leq Cd^{2-2\alpha}\sum_{k=1}^{\floor{dx}}k^{2(\alpha-2)+\theta}  \leq Cd^{\theta-1},
	\end{equation*}
	at which point we can take the infimum of this estimate over
	all $\theta \in (0,1)$ to get $ E_1\leq \frac{C}{d}$. For
	$E_2$, in a similar fashion, we apply Lemma
	\ref{covarianceLemmarosenblatt} to obtain $E_2\leq
	Cd^{3-2\alpha}$. Keeping in mind that $\frac{3}{2}<\alpha<2$
	is equivalent to $-1<3-2\alpha<0$,
	\eqref{estimateWassdistboundbyL2norm} implies that
	\begin{equation}
		\label{estimaterosenblatt}
		d_{W}\brac{\widehat{W}_{n,\floor{dx}},\Ro_n}\leq Cnd^{\frac{3-2\alpha}{2}}.
	\end{equation}
	Now to conclude that $\widehat{W}_{n,\floor{dx}}$ and $\Ro_n$ are
	close with respect to finite-dimensional distributions, we need a
	slightly different version of
	\eqref{estimateWassdistboundbyL2norm}, which can for example
	be found in \cite{Br}. Let $(x_1,\ldots,x_p)$ be a
	finite sequence in $[a,b]^{p}$, then we have 
	\begin{align*}
		d_{W}\brac{\brac{\widehat{W}_{n,\floor{dx_1}},\ldots,\widehat{W}_{n,\floor{dx_p}}}, \brac{\Ro_n,\ldots,\Ro_n}}\leq &\brac{\sum_{\substack{1\leq i\leq j\leq n\\ 1\leq l\leq p}}\E{\brac{\widehat{W}_{ij}(\floor{dx_l})-R_{ij}}^2}}^{\frac{1}{2}},
	\end{align*}
	for which the same estimate as in \eqref{estimaterosenblatt} clearly holds. 
	
	\section{Proof of the main functional convergence result}
	\noindent This section is dedicated to proving the main functional
	convergence result, Theorem \ref{theoremfunctionalnonquantitative}, which provides a functional
	counterpart to Theorems \ref{theoremcentralconvergence}, \ref{theoremthreehalf} and \ref{theoremrosenblatt}.
	\subsection{Proof of theorem \ref{theoremfunctionalnonquantitative}}
	\label{nonquantitativefunctionalsection}
	Our goal here is to show that our sequence of Wishart matrices
	converges in $C([a,b];M_n(\R))$ (without providing an estimate on the
	convergence rate). Recall from the introduction that $a<b$ are positive constants. Joint convergence of a vector in $C([a,b];M_n(\R))$
	is equivalent to marginal convergence of each component in
	$C([a,b];\R)$ per \cite[Theorem 26.23]{Da}, so that
	we only have to prove functional convergence of each entry of our
	matrix. Furthermore, it is a well-known fact that the condition 
	\begin{align*}
		\norm{W_{ij}\brac{\floor{dx}}-W_{ij}(\floor{dy})}_{L^p(\Omega)}\leq C_p\abs{x-y}^{\frac{1}{2}}
	\end{align*}
	for some $p>2$ and any $x,y\in [a,b]$, combined with convergence of
	finite dimensional distributions is sufficient in order to guarantee
	tightness in $C([a,b];\R)$. Convergence in the sense of finite
	dimensional distributions has already been shown in Theorems
	\ref{theoremcentralconvergence}, \ref{theoremthreehalf} and
	\ref{theoremrosenblatt}, so that we only need to verify the above
	condition ensuring tightness. In what follows, we reuse the notation
	and terminology introduced at the beginning of the proof of Theorem
	\ref{theoremcentralconvergence} in Subsection \ref{centralconvergencesection}.
	\\~\\
	\noindent  {\bf Case 1 ($\alpha<3/2$).} Note that by
        definition of multiple Wiener integrals (see
        \cite[Definition 2.7.1]{NP}), 
\begin{equation*}
I_2\brac{e_{ik}\otimes e_{jk}}
        = I_2\brac{e_{ik}\widetilde{\otimes} e_{jk}} =
        \frac{1}{2}\delta^2(e_{ik}\otimes e_{jk} +e_{jk}\otimes
        e_{ik}).
\end{equation*} 
Now, assume $y\leq x$ (without loss of generality) for $x,y\in [a,b]$. Apply the above transform
	together with Meyer's inequality \eqref{meyerineq} to get
	\begin{align*}
		\norm{W_{ij}\brac{\floor{dx}}-W_{ij}(\floor{dy})}_{L^p(\Omega)}& = \norm{\delta^2\brac{\frac{1}{2}\frac{1}{\sqrt{d}}\sum_{k=\floor{dy}}^{\floor{dx}}(e_{ik}\otimes e_{jk} +e_{jk}\otimes e_{ik})}}_{L^p(\Omega)}\\
		&\leq C\sum_{m=0}^2 \norm{D^m\brac{\frac{1}{2}\frac{1}{\sqrt{d}}\sum_{k=\floor{dy}}^{\floor{dx}}(e_{ik}\otimes e_{jk} +e_{jk}\otimes e_{ik})}}_{L^p(\Omega;\mathfrak{H}^{\otimes m+2})}\\
		&=C\norm{\frac{1}{2}\frac{1}{\sqrt{d}}\sum_{k=\floor{dy}}^{\floor{dx}}(e_{ik}\otimes e_{jk} +e_{jk}\otimes e_{ik})}_{\mathfrak{H}^{\otimes 2}}\\
		&=
		C\brac{d^{-1}\sum_{k=\floor{dy}}^{\floor{dx}}\delta_{kl}^2}^{\frac{1}{2}}.
	\end{align*}
	By Remark \ref{remarkafterlemma5} and the fact that $a_\alpha(k-l)\leq C\abs{k-l}^{\alpha-2}$ as $\abs{k-l}\to\infty$, we have
	\begin{align*}  
		\norm{W_{ij}\brac{\floor{dx}}-W_{ij}(\floor{dy})}_{L^p(\Omega)}&\leq C\brac{d^{-1}\sum_{k,l=\floor{dy}}^{\floor{dx}}\brac{\frac{k\wedge l}{k\vee l}}^{2\beta-\alpha}\abs{k-l}^{2(\alpha-2)}}^{\frac{1}{2}}\\ &\leq C\brac{d^{-1}\sum_{k,l=\floor{dy}}^{\floor{dx}}\abs{k-l}^{2(\alpha-2)}}^{\frac{1}{2}}\\
		&= C\brac{d^{-1}\sum_{m=k-l=1}^{\floor{dx}-\floor{dy}}\brac{\floor{dx}-\floor{dy}-m} m^{2(\alpha-2)}}^{\frac{1}{2}}\\
		&\leq C\brac{\sum_{m\in \mathbb{Z}}
			\abs{m}^{2(\alpha-2)}}^{\frac{1}{2}}(x-y)^{\frac{1}{2}} \leq C(x-y)^{\frac{1}{2}},
	\end{align*}
	which is the desired estimate.
	\\~\\
	\noindent  {\bf Case 2 ($\alpha=3/2$).} The same procedure as in the
	previous case gives
	\begin{align*}
		\norm{\widetilde{W}_{ij}\brac{\floor{dx}}-\widetilde{W}_{ij}(\floor{dy})}_{L^p(\Omega)}
		&\leq C\brac{(d\ln d)^{-1}\sum_{k,l=\floor{dy}}^{\floor{dx}}\abs{k-l}^{-1}}^{\frac{1}{2}}\\
		&= C\brac{(d\ln d)^{-1}\sum_{m=1}^{\floor{dx}-\floor{dy}}\brac{\floor{dx}-\floor{dy}-m} m^{-1}}^{\frac{1}{2}}\\
		&\leq C\brac{(\ln d)^{-1}\sum_{m=1}^{\floor{dx}-\floor{dy}} \abs{m}^{-1}}^{\frac{1}{2}}(x-y)^{\frac{1}{2}}
		\leq C(x-y)^{\frac{1}{2}},
	\end{align*}
	which gives the desired result.
	\\~\\
	\noindent  {\bf Case 3 ($\alpha>3/2$).} For this final case, the above
	argument yields
	\begin{align*}
		\norm{\widehat{W}_{ij}\brac{\floor{dx}}-\widehat{W}_{ij}(\floor{dy})}_{L^p(\Omega)}&\leq C\brac{d^{2-2\alpha}\sum_{k,l=\floor{dy}}^{\floor{dx}}\abs{k-l}^{2(\alpha-2)}}^{\frac{1}{2}}\\
		&= C\brac{d^{2-2\alpha}\sum_{m=1}^{\floor{dx}-\floor{dy}}\brac{\floor{dx}-\floor{dy}-m} m^{2(\alpha-2)}}^{\frac{1}{2}}\\
		&\leq
		C\brac{\int_0^{1}u^{2(\alpha-2)}du}^{\frac{1}{2}}(x-y)^{\frac{1}{2}}
		\leq C(x-y)^{\frac{1}{2}},
	\end{align*}
	which concludes the proof.
	
	\section{Technical Lemmas}
	\noindent This section gathers technical Lemmas used repeatedly in the
	proofs of our main results. For convenience, we group these auxiliary
	results by what proof they are related to. The notation used in all
	the results below is the one prevailing in Section \ref{introandmainresults}.
	\label{sectionLemma}
	\subsection{Lemmas related to the proof of Theorem
		\ref{theoremcentralconvergence} in Subsection \ref{centralconvergencesection}}
	\begin{Lemma}
		\label{covarianceLemmanormal}
		The covariance structure of the half-matrix $\W_{n,\floor{dx}}^{\operatorname{half}}$ is given by
		\begin{equation*}
			\begin{cases}
				\displaystyle \E{W_{il}W_{il}}=d^{-1}\sum_{k,j=1}^{\floor{dx}}\delta_{kj}^2 &
				\mbox{for }i\neq l\\
				\displaystyle \E{W_{ii}W_{ii}}=2d^{-1}\sum_{k,j=1}^{\floor{dx}}\delta_{kj}^2\\
				\displaystyle \E{W_{il}W_{mn}}=0 &\mbox{otherwise}
			\end{cases}.
		\end{equation*}
		Thus, if we denote $\mathcal{C}$ the covariance matrix of
		$\W_{n,\floor{dx}}^{\operatorname{half}}$, then $\mathcal{C}$ is
		diagonal with diagonal entries given by either $\E{W_{il}W_{il}}$ or $\E{W_{ii}W_{ii}}$.
	\end{Lemma}
	\begin{proof}
		For any $1 \leq i,l,m,n \leq n$, recalling the representation
		\eqref{representationofwij} of $W_{il}$, it holds that
		\begin{align*}
			\E{W_{il}W_{mn}}&=\E{I_2 \brac{\frac{1}{2\sqrt{d}}\sum_{k=1}^{\floor{dx}}e_{ik}\otimes e_{lk}+e_{lk}\otimes e_{ik}} I_2 \brac{\frac{1}{2\sqrt{d}}\sum_{j=1}^{\floor{dx}}e_{mj}\otimes e_{nj}+e_{nj}\otimes e_{mj}}}\\
			&=2!\left\langle
			\frac{1}{2\sqrt{d}}\sum_{k=1}^{\floor{dx}}\brac{e_{ik}\otimes e_{lk}+e_{lk}\otimes e_{ik}},
			\frac{1}{2\sqrt{d}}\sum_{j=1}^{\floor{dx}}\brac{e_{mj}\otimes e_{nj}+e_{nj}\otimes e_{mj}} \right\rangle_{\mathfrak{H}^{\otimes 2}} \\
			&=
			d^{-1}\sum_{k,j=1}^{\floor{dx}}\inner{e_{ik},e_{mj}}_{\mathfrak{H}}\inner{e_{lk},e_{nj}}_{\mathfrak{H}}
			+ d^{-1}\sum_{k,j=1}^{\floor{dx}}\inner{e_{ik},e_{nj}}_{\mathfrak{H}}\inner{e_{lk},e_{mj}}_{\mathfrak{H}}.
		\end{align*}
		This shows that the only entries of the matrix $\mathcal{C}$ that are non-zero
		are the ones for which $i=m$ and $l=n$ (note that we cannot encounter
		the case $i=n$ and
		$l=m$ as we are
		working with the half-matrix $\W_{n,\floor{dx}}^{\operatorname{half}}$). This corresponds to entries of the form $\E{W_{il}W_{il}}$. We hence only have to compute these entries and
		show that they are indeed equal to what is stated in the lemma. We
		can write
		\begin{align*}
			\E{W_{il}W_{il}}&=\E{I_2 \brac{\frac{1}{2\sqrt{d}}\sum_{k=1}^{\floor{dx}}e_{ik}\otimes e_{lk}+e_{lk}\otimes e_{ik}} I_2 \brac{\frac{1}{2\sqrt{d}}\sum_{j=1}^{\floor{dx}}e_{ij}\otimes e_{lj}+e_{lj}\otimes e_{ij}}}\\
			&=2!\left\langle
			\frac{1}{2\sqrt{d}}\sum_{k=1}^{\floor{dx}}\brac{e_{ik}\otimes
				e_{lk}+e_{lk}\otimes e_{ik}},
			\frac{1}{2\sqrt{d}}\sum_{j=1}^{\floor{dx}}\brac{e_{ij}\otimes
				e_{lj}+e_{lj}\otimes e_{ij}} \right\rangle_{\mathfrak{H}^{\otimes 2}} \\
			&=
			d^{-1}\sum_{k,j=1}^{\floor{dx}}\inner{e_{ik},e_{ij}}_{\mathfrak{H}}\inner{e_{lk},e_{lj}}_{\mathfrak{H}}
			+ d^{-1}\sum_{k,j=1}^{\floor{dx}}\inner{e_{ik},e_{lj}}_{\mathfrak{H}}\inner{e_{lk},e_{ij}}_{\mathfrak{H}} \\
			&=
			\begin{cases}
				\displaystyle 2d^{-1}\sum_{k,j=1}^{\floor{dx}}\delta_{kj}^2 & \mbox{if } i = l \\
				\displaystyle d^{-1}\sum_{k,j=1}^{\floor{dx}}\delta_{kj}^2 & \mbox{if } i \neq l
			\end{cases},
		\end{align*}
		as claimed.
	\end{proof}
	
	\begin{Lemma}
		\label{secondchaosLemma}
		For any $1 \leq i,j \leq n$, $W_{ij}$ belongs to the second Wiener chaos of
		the isonormal Gaussian process $X$ and has the representation \eqref{representationofwij}
		as a double Wiener integral, so that $W_{ij} =
		I_2(f_{ij})$, where $f_{ij} \in \mathfrak{H}^{\odot 2}$. Furthermore, for any $1 \leq i,j,m,n \leq n$, it holds that
		\begin{equation*}
			\E{\brac{\E{W_{ij}W_{mn}}-\frac{1}{2}\inner{DW_{ij},DW_{mn}}_{\mathfrak{H}}}^2}= 8\norm{f_{ij}\widetilde{\otimes}_{1}f_{mn}}^2_{\mathfrak{H}^{\otimes 2}}.
		\end{equation*}
	\end{Lemma}
	\begin{proof}
		Using the product formula \eqref{prodformula} together the stochastic
		Fubini theorem, it is straightforward
		to check that $\inner{DW_{ij},DW_{mn}}_\mathfrak{H} = 4 I_2
		\left(f_{ij}\widetilde{\otimes}_{1}f_{mn}  \right) + 4\left\langle
		f_{ij},f_{mn} \right\rangle_{\mathfrak{H}^{\otimes 2}}$, and hence
		deduce that $\E{W_{ij}W_{mn}} =
		\frac{1}{2}\E{\inner{DW_{ij},DW_{mn}}_\mathfrak{H}}$, so that
		\begin{align*}
			\E{\brac{\E{W_{ij}W_{mn}}-\frac{1}{2}\inner{DW_{ij},DW_{mn}}_{\mathfrak{H}}}^2}&=\E{\brac{\frac{1}{2}\E{\inner{DW_{ij},DW_{mn}}_\mathfrak{H}}-\frac{1}{2}\inner{DW_{ij},DW_{mn}}_{\mathfrak{H}}}^2}\\
			&=\frac{1}{4}\V{\inner{DW_{ij},DW_{mn}}_{\mathfrak{H}}}\\
			&=4\V{I_2\brac{f_{ij}\widetilde{\otimes}_{1}f_{mn}}+\left\langle
				f_{ij},f_{mn} \right\rangle_{\mathfrak{H}^{\otimes 2}}}\\
			&=4\E{I_2\brac{f_{ij}\widetilde{\otimes}_{1}f_{mn}}^2}\\
			&=8\norm{f_{ij}\widetilde{\otimes}_{1}f_{mn}}^2_{\mathfrak{H}^{\otimes 2}}.
		\end{align*}
	\end{proof}
	\begin{Lemma}
		\label{Lemmacontraction}
		For $f_{ij}$ and $f_{rs}$ as defined in
		\eqref{definitionoffij}, it holds that for any $1\leq i,j,r,s \leq n$,
		\begin{equation*}
			\norm{f_{ij}\widetilde{\otimes}_{1}f_{rs}}^2_{\mathfrak{H}^{\otimes
					2}}=\frac{1}{\tau d^2}\sum_{k,l,m,p=1}^{\floor{dx}}\delta_{kl}\delta_{mp}\delta_{km}\delta_{lp},
		\end{equation*}
		where $\tau =1$ if $i=j=r=s$, $\tau = 4$ if $i=j=r\neq s$ or
		$i=j=s\neq r$ or $i=r=s\neq j$ or $j=r=s\neq i$, $\tau = 8$ if $i=r\neq s,\ j=s\neq i$ or $i=s\neq j,\ j=r\neq s$ or $j=r\neq i,\
		i=s\neq r$ and $\tau = 16$ if $i=r,\ i\neq j,\ s\neq j,\ r\neq s$ or $i=s,\ i \neq j,\ i \neq r,\
		r\neq j$ or $j=r,\ j\neq i,\ s\neq i,\ r\neq s$ or $j=s,\ j\neq i,
		j\neq r,\ r\neq i$. In all other cases,
		\begin{equation*}
			\norm{f_{ij}\widetilde{\otimes}_{1}f_{rs}}^2_{\mathfrak{H}^{\otimes
					2}}=0.
		\end{equation*}
	\end{Lemma}
	\begin{proof}
		Recalling the definition of $f_{ij}$ given in \eqref{definitionoffij},
		we can write, for any $1\leq i,j,r,s \leq n$,
		\begin{align*}
			\norm{f_{ij}\widetilde{\otimes}_{1}f_{rs}}^{2}_{\mathfrak{H}^{\otimes 2}}&=\norm{\frac{1}{2\sqrt{d}}\sum_{k=1}^{\floor{dx}}\brac{e_{ik}\otimes e_{jk}+e_{jk}\otimes e_{ik}}   \widetilde{\otimes}_{1} \frac{1}{2\sqrt{d}}\sum_{l=1}^{\floor{dx}}\brac{e_{rl}\otimes e_{sl}+e_{sl}\otimes e_{rl}}}^{2}_{\mathfrak{H}^{\otimes 2}}\\
			&=\frac{1}{16d^2}\Bigg\Vert \sum_{k,l=1}^{\floor{dx}}\left( e_{jk}\otimes
			e_{sl}\delta_{kl}\mathds{1}_{\left\{ i=r \right\}} + e_{jk}\otimes
			e_{rl}\delta_{kl}\mathds{1}_{\left\{ i=s \right\}} \right.\\
			& \left. \qquad\qquad\qquad\qquad\qquad\qquad\quad + e_{ik}\otimes
			e_{sl}\delta_{kl}\mathds{1}_{\left\{ j=r \right\}} + e_{ik}\otimes
			e_{rl}\delta_{kl}\mathds{1}_{\left\{ j=s \right\}}  \right)\Bigg\Vert^{2}_{\mathfrak{H}^{\otimes 2}}\\
			&=\frac{1}{16d^2}\sum_{k,l,m,p=1}^{\floor{dx}}\delta_{kl}\delta_{mp}\delta_{km}\delta_{lp}
			\left(\mathds{1}_{\left\{
				i=r \right\}} + \mathds{1}_{\left\{ i=r=s \right\}}  +
			\mathds{1}_{\left\{ i=j=r \right\}} + \mathds{1}_{\left\{ i=j=r=s
				\right\}} \right. \\
			&\left. \qquad\qquad\qquad\qquad\qquad\qquad\ + \mathds{1}_{\left\{
				i=s \right\}} + \mathds{1}_{\left\{ i=r=s \right\}}  +
			\mathds{1}_{\left\{ i=j=s \right\}} + \mathds{1}_{\left\{ i=j=r=s
				\right\}} \right.\\
			&\left. \qquad\qquad\qquad\qquad\qquad\qquad\ + \mathds{1}_{\left\{
				j=r \right\}} + \mathds{1}_{\left\{ j=r=s \right\}}  +
			\mathds{1}_{\left\{ j=i=r \right\}} + \mathds{1}_{\left\{ i=j=r=s
				\right\}} \right.\\
			&\left. \qquad\qquad\qquad\qquad\qquad\qquad\ + \mathds{1}_{\left\{
				j=s \right\}} + \mathds{1}_{\left\{ j=r=s \right\}}  +
			\mathds{1}_{\left\{ j=i=s \right\}} + \mathds{1}_{\left\{ i=j=r=s
				\right\}} \right),
		\end{align*}
		from which the conclusion follows easily.
	\end{proof}
	\noindent The following Lemma is borrowed from \cite{HN} and provides us with
	the asymptotic behaviour of the variance of the entries of $\W_{n,\floor{dx}}$.
	\begin{Lemma}
		\label{varianceLemma}
		Denote by $\sigma^2_{\floor{dx}}$ the variance of the entries $W_{ij},i\neq j$ of $\W_{n,\floor{dx}}$. Let $\alpha < \frac{3}{2}$ and define $a_\alpha (m)=\frac{1}{2}\brac{\abs{m+1}^{\alpha}+\abs{m-1}^{\alpha}-2\abs{m}^{\alpha}}$. Then,
		\begin{align*}
			\sigma^2_{\floor{dx}}=\frac{1}{d}\sum_{k,l=1}^{\floor{dx}}\abs{\delta_{kl}}^2
		\end{align*}
		and
		\begin{align*}
			\sigma^2=\lim_{d\to\infty} \sigma^2_{\floor{dx}}=\frac{x}{2}\sum_{m\in\mathbb{Z}}a_\alpha (m)^2.
		\end{align*}
	\end{Lemma}
	\begin{proof}
		Refer to the proof of \cite[Lemma 5.1]{HN}. 
	\end{proof}
	\begin{remark}
		\label{remarkafterlemma5}
		An important observation coming from the proof of \cite[Lemma 5.1]{HN}
		is that
		\begin{align*}
			\sigma^2_{\floor{dx}}=\brac{\frac{1}{d}\sum_{m=1}^{\floor{dx}-1}\sum_{k=1}^{\floor{dx}-m-1}\brac{\frac{k}{k+m}}^{2\beta-\alpha}a^2_\alpha(m)}+\frac{x}{2}a_\alpha (0)+R_d,
		\end{align*}
		where $R_d$ is a remainder term with $R_d=o(d^{2\alpha-3}+d^{-1})$.
	\end{remark}
	\begin{Lemma}
		\label{fbmcomparisonLemma}
		Assume that $\alpha<1$ and $\alpha+\nu>2$, or that $1\leq
		\alpha<\frac{3}{2}$. Then, it holds that
		\begin{align*}
			\frac{1}{d^2}\sum_{k,l,m,p=1}^{\floor{dx}}\delta_{kl}\delta_{mp}\delta_{km}\delta_{lp}\leq 
			\frac{C}{d^2}\sum_{k,l,m,p=1}^{\floor{dx}}\abs{a_\alpha(k-l)a_\alpha(m-p)a_\alpha(k-m)a_\alpha(l-p)},
		\end{align*}
		where as in Lemma \ref{varianceLemma}, $a_\alpha
		(i)=\frac{1}{2}\brac{\abs{i+1}^{\alpha}+\abs{i-1}^{\alpha}-2\abs{i}^{\alpha}}$. As
		a result, we have
		\begin{align*}
			\frac{1}{d^2}\sum_{k,l,m,p=1}^{\floor{dx}}\delta_{kl}\delta_{mp}\delta_{km}\delta_{lp}\leq 
			\frac{C}{d}\brac{\sum_{m=-\floor{dx}+1}^{\floor{dx}-1}\abs{a_\alpha (m)}^{\frac{4}{3}}}^3.
		\end{align*}
	\end{Lemma}
	\begin{proof}
		In order to obtain the first estimate, it is sufficient to show that
		$\abs{\delta_{kl}}\leq C\abs{a_\alpha(k-l)}$ for any
		$1 \leq k$ and $l \leq
		\floor{dx}$. By symmetry, we also only need to examine the case
		where $l \leq k$, which we separate into three separate cases. If
		$\ceil{\frac{k}{3}}\leq l\leq k-2$, which implies that $k-l\leq 2l$,
		then \cite[Lemma 3.1, Lemma 3.2, Part (b)]{HN} implies that
		\begin{align*}
			\abs{\delta_{kl}}&\leq C(lk)^{\alpha/2-\beta}\brac{l^{2\beta-\alpha}a_\alpha(k-l)+l^{2\beta-\alpha-1}(k-l)^{\alpha-1}+l^{2\beta-2}}\\
			&\leq C\brac{a_\alpha(k-l)+l^{-1}(k-l)^{\alpha-1}+l^{\alpha-2}}.
		\end{align*}
		Since $a_\alpha(k-l)=\frac{1}{2}\alpha(\alpha-1)(k-l)^{\alpha-2}+o\brac{(k-l)^{\alpha-2}}$, it follows that for $k-l\leq 2l$, $l^{-1}(k-l)^{\alpha-1}\leq Ca_\alpha(k-l)$ and $l^{\alpha-2}\leq Ca_\alpha(k-l)$. Thus,
		\begin{align*}
			\abs{\delta_{kl}}\leq C\abs{a_\alpha(k-l)}.
		\end{align*}
		If $1\leq l<\ceil{\frac{k}{3}}$, which implies that $2l\leq k-l+3<C(k-l)$, then
		\cite[Lemma 3.1, Lemma 5.1]{HN} yields, whenever $\alpha+\nu>2$,
		\begin{align*}
			\abs{\delta_{kl}}&\leq C(lk)^{\frac{\alpha}{2}-\beta}\brac{\brac{l^{2\beta+\nu-2}(k-l)^{-\nu}}\vee\brac{l^{2\beta-\alpha}(k-l)^{\alpha-2}}} \\
			&\leq C\brac{\brac{l^{\alpha+\nu-2}(k-l)^{-\nu}}\vee(k-l)^{\alpha-2}}\\
			&\leq C\abs{a_\alpha(k-l)}.
		\end{align*}
		If $k=l$, then $\delta_{lk}=a_\alpha(l-k)=1$. Now,
		consider the case when $l=k-1$. Since
		$a_\alpha(1)=2^{\alpha-1}-1$, \cite[Lemma 3.1, Lemma 3.2,
		Part (a)]{HN} implies that there exists a constant
		$C>0$ such that for $k$ large enough, one has
		\begin{equation*}
			\abs{\delta_{(k-1)k}} \leq C\brac{\frac{k}{k-1}}^{\beta-\frac{\alpha}{2}} \leq C\abs{a_\alpha(1)}.
		\end{equation*}
		Combining all three of these cases yields
		\begin{align*}
			\frac{1}{d^2}\sum_{k,l,m,p=1}^{\floor{dx}}\delta_{kl}\delta_{mp}\delta_{km}\delta_{lp}
			&\leq \frac{1}{d^2}\sum_{k,l,m,p=1}^{\floor{dx}}\abs{\delta_{kl}\delta_{mp}\delta_{km}\delta_{lp}}\\
			&\leq \frac{C}{d^2}\sum_{k,l,m,p=1}^{\floor{dx}}\abs{a_\alpha(k-l)a_\alpha(m-p)a_\alpha(k-m)a_\alpha(l-p)}.
		\end{align*}
		The second estimate is due to a result in \cite[Pages 134--135]{NP},
		which states
		\begin{align*}
			&\frac{C}{d^2}\sum_{k,l,m,p=1}^{\floor{dx}}a_\alpha(k-l)a_\alpha(m-p)a_\alpha(k-m)a_\alpha(l-p)\\
			& \qquad \leq
			\frac{C}{d^2}\sum_{k,l,m,p=1}^{\floor{dx}}\abs{a_\alpha(k-l)a_\alpha(m-p)a_\alpha(k-m)a_\alpha(l-p)}
			\leq\frac{C}{d}\brac{\sum_{m=-\floor{dx}+1}^{\floor{dx}}\abs{a_\alpha
					(m)}^{\frac{4}{3}}}^3.
		\end{align*}
		Finally, note that the fact that we impose $\alpha<1$ and $\alpha+\nu>2$, or that $1\leq
		\alpha<\frac{3}{2}$ is due to our hypothesis that $1<\nu\leq 2$.
	\end{proof}

	\subsection{Lemmas related to the proof of Theorem
		\ref{theoremthreehalf} in Subsection \ref{threehalfsection}}
	\begin{Lemma}
		\label{fbmcomparisonLemmathreehalf}
		Assume that $\alpha=\frac{3}{2}$. Let $a_\alpha (m)=\frac{1}{2}\brac{\abs{m+1}^{\alpha}+\abs{m-1}^{\alpha}-2\abs{m}^{\alpha}}$. Then, 
		\begin{align*}
			\frac{1}{(d\ln{d})^2}\sum_{k,l,m,p=1}^{\floor{dx}}\delta_{kl}\delta_{mp}\delta_{km}\delta_{lp}\leq 
			\frac{C}{(d\ln{d})^2}\sum_{k,l,m,p=1}^{\floor{dx}}\abs{a_\alpha(k-l)a_\alpha(m-p)a_\alpha(k-m)a_\alpha(l-p)}.
		\end{align*}
		As a result, 
		\begin{align*}
			\frac{1}{(d\ln{d})^2}\sum_{k,l,m,p=1}^{\floor{dx}}\delta_{kl}\delta_{mp}\delta_{km}\delta_{lp}\leq 
			\frac{C}{d(\ln{d})^2}\brac{\sum_{m=-\floor{dx}+1}^{\floor{dx}-1}\abs{a_\frac{3}{2} (m)}^{\frac{4}{3}}}^3.
		\end{align*}
	\end{Lemma}
	\begin{proof}
		The proof follows in the exact same way as the proof of Lemma \ref{fbmcomparisonLemma}.
	\end{proof}
	\begin{Lemma}
		\label{varianceLemmathreehalf}
		Denote by $\rho^2_{\floor{dx}}$ the variance of the non-diagonal entries of $\widetilde{\W}_{n,\floor{dx}}$. Then,
		\begin{equation*}
			\rho^2_{\floor{dx}}=\frac{1}{d\ln{d}}\sum_{k,l=1}^{\floor{dx}}\abs{\delta_{kl}}^2
		\end{equation*}
		and
		\begin{equation*}
			\rho^2=\lim_{d\to\infty} \rho^2_{\floor{dx}}= \frac{9x}{32}.
		\end{equation*}
	\end{Lemma}
	\begin{proof}
		We adapt the ideas in \cite[Proof of Lemma 5.1]{HN} which does not cover the case $\alpha=\frac{3}{2}$.
		\begin{proofstep}
			Define $\xi_{j,d}=\norm{\Delta X_{\frac{j}{d}}}_{L^2(\Omega)}$. Choose $\gamma\in (0,\frac{1}{2})$ and let $\tau=(\floor{dx})^\gamma$. We will perform the decomposition
			\begin{equation*}
				\rho^2_{\floor{dx}}=\frac{1}{d\ln{d}}\sum_{k,l=1}^{\floor{dx}}\abs{\delta_{kl}}^2=A_{1,d}+A_{2,d}+A_{3,d}+A_{4,d},
			\end{equation*}
			where
			\begin{align*}
				A_{1,d}&=\frac{1}{d\ln{d}}\sum_{j\in D_1,\ k\in D_1}\xi_{j,d}^{-2}\xi_{k,d}^{-2}\brac{\E{\Delta X_{\frac{j}{d}}\Delta X_{\frac{k}{d}}}}^{2},\\
				A_{2,d}&=\frac{1}{d\ln{d}}\sum_{j\in D_2,\ k\in D_2}\xi_{j,d}^{-2}\xi_{k,d}^{-2}\brac{\E{\Delta X_{\frac{j}{d}}\Delta X_{\frac{k}{d}}}}^{2},\\
				A_{3,d}&=\frac{1}{d\ln{d}}\sum_{j\in D_1,\ k\in D_2}\xi_{j,d}^{-2}\xi_{k,d}^{-2}\brac{\E{\Delta X_{\frac{j}{d}}\Delta X_{\frac{k}{d}}}}^{2},\\
				A_{4,d}&=\frac{1}{d\ln{d}}\sum_{j\in D_2,\ k\in
					D_1}\xi_{j,d}^{-2}\xi_{k,d}^{-2}(\E{\Delta
					X_{\frac{j}{d}}\Delta X_{\frac{k}{d}}})^{2},
			\end{align*}
			with
			\begin{align*}
				D_1&=\{l\colon 1\leq l\leq  \tau \wedge \floor{dx} \}\\
				D_2&=\{l\colon \tau< l\leq  \floor{dx} \}.
			\end{align*}
			$A_{1,d}$ is bounded by $Cd^{2\gamma-1}(\ln{d})^{-1}$ by the
			Cauchy-Schwarz inequality, so that it converges to zero as $d$
			goes to infinity. 
		\end{proofstep}
		\begin{proofstep}
			We further decompose and bound $A_{2,d}$ and
			$A_{4,d}$. $A_{3,d}$ can be bounded in the same way as
			$A_{4,d}$. We have
			\begin{align*}
				A_{2,d}&= \frac{1}{d\ln{d}}\brac{\floor{dx}
					-\tau}+\frac{1}{d\ln{d}}\sum_{j,k\in
					D_2 \colon \abs{j-k}=1}\xi_{j,d}^{-2}\xi_{k,d}^{-2}\brac{\E{\Delta X_{\frac{j}{d}}\Delta X_{\frac{k}{d}}}}^{2}\\
				&\quad +\frac{1}{d\ln{d}}\sum_{j,k\in D_2\colon \abs{j-k}\geq 2}\xi_{j,d}^{-2}\xi_{k,d}^{-2}\brac{\E{\Delta X_{\frac{j}{d}}\Delta X_{\frac{k}{d}}}}^{2}\\
				&=B_{1,d}^{(1)}+B_{2,d}^{(1)}+B_{3,d}^{(1)}
			\end{align*}
			and
			\begin{align*}
				A_{4,d}&= \frac{1}{d\ln{d}}\xi_{\floor{\tau},d}^{-2}\xi_{\ceil{\tau},d}^{-2}\brac{\E{\Delta X_{\frac{\floor{\tau}}{d}}\Delta X_{\frac{\ceil{\tau}}{d}}}}^{2}\\
				&\quad +\frac{1}{d\ln{d}}\sum_{j\in D_2,\ k\in D_1\colon \abs{j-k}\geq 2}\xi_{j,d}^{-2}\xi_{k,d}^{-2}\brac{\E{\Delta X_{\frac{j}{d}}\Delta X_{\frac{k}{d}}}}^{2}\\
				&=B_{1,d}^{(2)}+B_{2,d}^{(2)}.
			\end{align*}
			$B_{1,d}^{(1)}$ clearly goes to $0$ as $d\to\infty$. $B_{1,d}^{(2)}$
			and $B_{2,d}^{(1)}$ also go to $0$ by the Cauchy-Schwarz
			inequality. For $B_{2,d}^{(1)}$ in particular, we can write
			\begin{align*}
				\frac{1}{d\ln{d}}\sum_{j,k\in D_2 \colon\abs{j-k}=1}\xi_{j,d}^{-2}\xi_{k,d}^{-2}\brac{\E{\Delta X_{\frac{j}{d}}\Delta X_{\frac{k}{d}}}}^{2}\leq \frac{1}{d\ln{d}}\brac{ \floor{dx} -\tau}.
			\end{align*}
		\end{proofstep}
		\begin{proofstep}
			In this step, we argue that all terms with $k<\floor{\frac{j}{3}}$
			have no contribution to $\rho^2$ as $d\to\infty$. The
			case $j\leq\floor{\frac{k}{3}}$ can be treated similarly. \cite[Lemma 5.1]{HN} gives the bound for $k<\floor{\frac{j}{3}}$
			\begin{align*}
				\brac{\E{\Delta X_{\frac{j}{d}}\Delta X_{\frac{k}{d}}}}^2 \leq Cd^{-4\beta}k^{4\beta-3}(j-k)^{-1}.
			\end{align*}
			Meanwhile, \cite[Lemma 3.1]{HN} states that
			\begin{align*}
				\xi_{j,d}^2 &=  2\lambda j^{2\beta-\frac{3}{2}}d^{-2\beta}(1+\eta_{j,d})\\
				\xi_{k,d}^2 &= 2\lambda k^{2\beta-\frac{3}{2}}d^{-2\beta}(1+\eta_{k,d}),
			\end{align*}
			such that $\eta_{j,d}\leq Cd^{-\gamma d}$ and $\eta_{k,d}\leq
			Cd^{-\gamma d}$. Hence, 
			\begin{align*}
				\frac{1}{d\ln{d}}\sum_{j=3}^{\floor{dx}}\sum_{k=1}^{\floor{\frac{j}{3}}}\xi_{j,d}^{-2}\xi_{k,d}^{-2}\brac{\E{\Delta X_{\frac{j}{d}}\Delta X_{\frac{k}{d}}}}^{2}&\leq \frac{C}{d\ln{d}}\sum_{j=3}^{\floor{dx}}\sum_{k=1}^{\floor{\frac{j}{3}}}(j-k)^{-1}\\&\leq 
				\frac{C}{d\ln{d}}\sum_{j=3}^{\floor{dx}}\int_1^{\floor{\frac{j}{3}}}(j-y)^{-1}dy
				\\&\leq \frac{C}{d\ln{d}}\sum_{j=3}^{\floor{dx}} \ln{\brac{1-\floor{1/3}}}\\
				&\leq \frac{C}{\ln d}.
			\end{align*}
		\end{proofstep}
		\begin{proofstep}
			In this step, we study those terms which belongs to $B_{3,d}^{(1)}$
			and $B_{2,d}^{(2)}$ and were not considered in Step $3$. In the case
			$k\leq j-2$, we use the covariance representation from \cite[Lemma 3.1
			and Lemma 3.2]{HN} in order to get
			\begin{align*}
				C_d&=\frac{1}{d\ln{d}}\sum_{\substack{j\in D_2\\\floor{\frac{j}{3}}\leq k\leq j-2}}\xi_{j,d}^{-2}\xi_{k,d}^{-2}\brac{\E{\Delta X_{\frac{j}{d}}\Delta X_{\frac{k}{d}}}}^{2}\\
				&= \frac{1}{d\ln{d}}\sum_{\substack{j\in D_2\\\floor{\frac{j}{3}}\leq k\leq j-2}}\brac{\brac{\frac{k}{j}}^{\beta-3/4}a_\frac{3}{2}(j-k)+R_{j,k}}^2\\
				&=\frac{1}{d\ln{d}}\sum_{\substack{j\in D_2\\\floor{\frac{j}{3}}\leq k\leq j-2}}\brac{\frac{k}{j}}^{2\beta-3/2}a^2_\frac{3}{2}(j-k)\\
				&\quad + \brac{\frac{2}{d\ln{d}}\sum_{\substack{j\in D_2\\\floor{\frac{j}{3}}\leq k\leq j-2}}\brac{\frac{k}{j}}^{\beta-3/4}a_\frac{3}{2}(j-k)R_{j,k}+R_{j,k}^2}\\
				&= D_d+O_d,
			\end{align*}
			where according to \cite[Lemma 3.2]{HN} and the fact that $\frac{j}{3}\leq k\leq j-2 \Rightarrow j-k\leq 2k$,
			\begin{equation*}
				R_{j,k}\leq C\brac{\frac{k}{j}}^{\beta-3/4}k^{-1}(j-k-1)^{\frac{1}{2}}+C\brac{\frac{k}{j}}^{\beta-3/4}k^{-\frac{1}{2}}
				\leq Ck^{-\frac{1}{2}}.
			\end{equation*}
			Now since $a_\alpha(m)=\frac{3}{8}m^{-1/2}+o\brac{m^{-1/2}}$, $O_d$ can be bounded via
			\begin{align*}
				O_d &\leq \frac{C}{d\ln{d}}\sum_{j=\ceil{\tau}}^{\floor{dx}}\sum_{k=\ceil{\frac{j}    {3}}}^{j-2} (j-k)^{-\frac{1}{2}}k^{-\frac{1}{2}}\\
				&\leq \frac{C}{d\ln{d}}\sum_{k=1}^{\floor{dx}-3}\sum_{j= k+2}^{3k} (j-k)^{-\frac{1}{2}}k^{-\frac{1}{2}}\leq \frac{C}{d},
			\end{align*}
			which vanishes as $d\to\infty$.
			
		\end{proofstep}
		\begin{proofstep}
			The last term $D_d$ is the only one with a non-trivial contribution to
			$\rho^2_{\floor{dx}}$ as $d\to\infty$. We will
			show that
			\begin{align}
				\label{limitthreehalf}
				\lim_{d\to\infty} D_d=\lim_{d\to\infty}\frac{1}{d\ln{d}}\sum_{j=3}^{\floor{dx}}\sum_{k=1}^{j-2}\brac{\frac{k}{j}}^{2\beta-\frac{3}{2}}a^2_{\frac{3}{2}}(j-k)= \frac{9x}{64}.
			\end{align}
			Since $a_{\frac{3}{2}}(m)=
			\frac{3}{8}m^{-\frac{1}{2}}+\delta_{\frac{3}{2}}(m)$, where $\delta_{\frac{3}{2}}(m)=o(m^{-\frac{1}{2}})$, it follows that
			\begin{align*}
				D_d^*&=\lim_{d\to\infty}\frac{1}{d\ln{d}}\sum_{j=3}^{\floor{dx}}\sum_{k=1}^{j-2}a^2_{\frac{3}{2}}(j-k)=\lim_{d\to\infty}\frac{1}{d\ln{d}}\sum_{m=j-k=1}^{\floor{dx}-1}\sum_{k=1}^{\floor{dx}-m-1}a^2_{\frac{3}{2}}(m)\\
				&=\lim_{d\to\infty}\frac{1}{\ln{d}}\sum_{m=1}^{\floor{dx}-1}\frac{\floor{dx}-m-1}{d}\brac{\frac{9}{64}m^{-1}+\frac{3}{4}m^{-\frac{1}{2}}\delta_{\frac{3}{4}}(m)+\delta^2_{\frac{3}{2}}(m)}\\
				&=\lim_{d\to\infty}\frac{9}{64\ln{d}}\sum_{m=1}^{\floor{dx}-1}\frac{\floor{dx}-m-1}{d}m^{-1}+E_d.
			\end{align*}
			The fact that $\lim_{m\to\infty}\frac{\delta_{3/2}(m)}{m^{-1/2}}=0$
			implies that for $\epsilon<1$, there exists $M_\epsilon\in \N$ such that
			$\delta_{3/2}(m)\leq \epsilon m^{-1/2}$ for all $m\geq M_\epsilon$. In
			addition, this means that $\delta_{3/2}(m)$ is bounded by some constant $C>1$. Hence,
			\begin{align*}
				&E_d=\lim_{d\to\infty}\frac{1}{\ln{d}}\sum_{m=1}^{\floor{dx}-1}\frac{\floor{dx}-m-1}{d}\brac{\frac{3}{4}m^{-\frac{1}{2}}\delta_{\frac{3}{2}}(m)+\delta^2_{\frac{3}{2}}(m)}\\
				&\leq \lim_{d\to\infty}\frac{1}{\ln{d}}\sum_{m=1}^{M_\epsilon-1}\brac{\frac{3}{4}m^{-\frac{1}{2}}\delta_{\frac{3}{2}}(m)+\delta^2_{\frac{3}{2}}(m)}
				+\lim_{d\to\infty}\frac{1}{\ln{d}}\sum_{m=M_\epsilon}^{\floor{dx}-1}\brac{\frac{3}{4}m^{-\frac{1}{2}}\delta_{\frac{3}{2}}(m)+\delta^2_{\frac{3}{2}}(m)}\\
				&\leq  \lim_{d\to\infty}\frac{1}{\ln d}\sum_{m=1}^{M_\epsilon-1}\brac{\frac{3}{4}m^{-\frac{1}{2}}\delta_{\frac{3}{2}}(m)+\delta^2_{\frac{3}{2}}(m)}
				+\lim_{d\to\infty}\frac{1}{\ln{d}}\sum_{m=M_\epsilon}^{\floor{dx}-1}\brac{\frac{3}{4}\epsilon m^{-1}+\epsilon^2 m^{-1}}\\
				&\leq  \lim_{d\to\infty}\frac{3CM_\epsilon}{\ln d}
				+\lim_{d\to\infty}\frac{2\epsilon}{\ln{d}}\sum_{m=M_\epsilon}^{\floor{dx}-1}m^{-1}.
			\end{align*}
			The first term in the above inequality is clearly $0$. For the second
			term, observe that
			\begin{align*}
				\lim_{d\to\infty}\frac{1}{\ln d}\sum_{m=1}^{\floor{dx}-1} m^{-1}=1,
			\end{align*}
			so that $\frac{1}{\ln d}\sum_{m=1}^{\floor{dx}-1} m^{-1}$
			is uniformly bounded by some constant $C_1$ for all $d \geq 1$. We
			hence deduce that, for all $d \geq 1$, $E_d\leq \epsilon C_1$
			which holds for all $\epsilon<1$. This implies $E_d=0$ and
			\begin{align*}
				D_d^*&=\lim_{d\to\infty}\frac{9}{64\ln{d}}\sum_{m=1}^{\floor{dx}-1}\frac{\floor{dx}-m-1}{d}m^{-1}\\
				&=\frac{9x}{64}\lim_{d\to\infty}\frac{1}{\ln{d}}\sum_{m=1}^{\floor{dx}-1}m^{-1}-\frac{9}{64}\lim_{d\to\infty}\frac{1}{d\ln{d}}\sum_{m=1}^{\floor{dx}-1}(1+m^{-1})
				=\frac{9x}{64}.
			\end{align*}
			Now, if $r_d=\abs{D_d^*-D_d}$ and $\lim_{d\to\infty}r_d= 0$, then \eqref{limitthreehalf} holds. To this end we have
			\begin{align*}
				r_d&=\frac{1}{d\ln{d}}\sum_{j=3}^{\floor{dx}}\sum_{k=1}^{j-2}\brac{1-\brac{\frac{k}{j}}^{2\beta-\frac{3}{2}}}a^2_{\frac{3}{2}}(j-k)\\
				&\leq \frac{1}{d\ln{d}}\sum_{m=1}^{\floor{dx}-1}a^2_{\frac{3}{2}}(m)\sum_{k=1}^{\floor{dx}-m-1}\brac{1-\brac{1-\frac{m}{k+m}}^{2\beta-\frac{3}{2}}}.
			\end{align*}
			Note that as $x\to\infty$, $\ln{\abs{x}}\leq \abs{x}^\zeta$ for any
			positive value of $\zeta$. Also, $a_{\frac{3}{2}}(m)\leq \frac{3}{4}m^{-\frac{1}{2}}$ as $m\to\infty$. Thus, a Taylor expansion of $(1-x)^p$ for $p>0$ and $0\leq x<1$ yields
			\begin{align*}
				r_d&\leq \frac{C}{d\ln{d}}\sum_{m=1}^{\floor{dx}-1}m^{-1}\sum_{k=1}^{\floor{dx}-m-1}\frac{m}{k+m}\\
				&\leq \frac{C}{d\ln d}\sum_{m=1}^{\floor{dx}-1}m^{-1}\int_1^{\floor{dx}-m-1}\frac{m}{y+m-1}dy\\
				&=   \frac{C}{d\ln d}\sum_{m=1}^{\floor{dx}-1} \ln{\frac{\floor{dx}}{m}}\\
				&\leq   \frac{C}{d\ln d}\sum_{m=1}^{\floor{dx}-1} \brac{\frac{\floor{dx}}{m}}^\zeta
				\leq \frac{C}{\ln d},
			\end{align*}
			which implies $\lim_{d\to\infty}r_d= 0$. Finally, combining the cases $k\leq j$ and $j<k$ yields 
			\begin{align*}
				\lim_{d\to\infty} \rho^2_{\floor{dx}}= \frac{9x}{32}.
			\end{align*}
		\end{proofstep}
	\end{proof}
	
	\subsection{Lemmas related to the proof of Theorem
		\ref{theoremrosenblatt} in Subsection \ref{proofoftheorosenblatt}}
	\begin{Lemma}
		\label{covarianceLemmarosenblatt}
		The covariance structure of $X$ can be written as a double integral as
		\begin{equation*}
			\E{(X_a-X_{a-\epsilon})(X_b-X_{b-\delta})}=\int_{b-\delta}^{b}\int_{a-\epsilon}^{a}\partial_{s,t}\E{X_s X_t}dsdt.
		\end{equation*}   
		Moreover, whenever $\frac{s\wedge t}{s\vee t}\leq \frac{1}{2}$, the
		following bound holds  
		\begin{equation*}    
			\abs{\partial_{s,t}\E{X_s X_t}}\leq C\brac{s\wedge t}^{2\beta-\alpha}(s\vee t)^{\alpha-2},
		\end{equation*}
		and whenever $\frac{1}{2}< \frac{s\wedge t}{s\vee t}< 1$, we have
		\begin{equation*}    
			\abs{\partial_{s,t}\E{X_s X_t}}\leq C\brac{s\wedge t}^{2\beta-\alpha}\abs{s-t}^{\alpha-2}.
		\end{equation*}
	\end{Lemma}
	\begin{proof}
		The first assertion can be deduced from writing
		\begin{align*}
			\E{(X_a-X_{a-\epsilon})(X_b-X_{b-\delta})}&= \E{X_aX_b}-\E{X_{a-\epsilon}X_b}-\E{X_aX_{b-\delta}}+\E{X_{a-\epsilon}X_{b-\delta}}\\
			&=\int_{b-\delta}^{b}\int_{a-\epsilon}^{a}\partial_{s,t}\E{X_s X_t}dsdt.
		\end{align*}
		The bounds on $\abs{\partial_{s,t}\E{X_s X_t}}$ are consequences of
		hypotheses \textbf{(H.1)} and \textbf{(H.2)} as we will show next. Without loss of
		generality, let's assume $s\leq t$. In the case $\frac{s}{t}\leq
		\frac{1}{2}$, \textbf{(H.2)} implies that
		\begin{align*}
			\partial_{s,t}\E{X_s X_t}
			&=(2\beta-1)s^{2\beta-2}\phi'\brac{\frac{t}{s}}-s^{2\beta-3}\phi''\brac{\frac{t}{s}}\\
			&\leq  Cs^{2\beta-2}\brac{\frac{t}{s}}^{\alpha-2}+Cs^{2\beta-3}\brac{\frac{t}{s}}^{\alpha-3}\\
			&\leq  Cs^{2\beta-\alpha}t^{\alpha-2}.
		\end{align*}
		Meanwhile, whenever $\frac{1}{2}<\frac{s}{t}\leq 1$, \textbf{(H.1)} implies that
		\begin{align*}
			\partial_{s,t}\E{X_s X_t}&=(1-2\beta)\lambda\alpha
			x^{2\beta-2}\brac{\frac{t}{s}-1}^{\alpha-1}+\lambda
			\alpha(\alpha-1)s^{2\beta-3}t\brac{\frac{t}{s}-1}^{\alpha-2}\\
			&\qquad\qquad\qquad\qquad\qquad\qquad\qquad\qquad\qquad +(2\beta-1)s^{2\beta-2}\phi'\brac{\frac{t}{s}}-x^{2\beta-3}t\phi''\brac{\frac{y}{s}}\\
			&\leq C s^{2\beta-2}\brac{\frac{t}{s}-1}^{\alpha-1}+Cs^{2\beta-3}t\brac{\frac{t}{s}-1}^{\alpha-2}+Cs^{2\beta-1}t^{\alpha-1}+Cs^{2\beta-2}\brac{\frac{t}{s}-1}^{\alpha-1}\\
			&\leq Cs^{2\beta-3}t\brac{\frac{t}{s}-1}^{\alpha-2}\\
			&=  Cs^{2\beta-\alpha-1}t(t-s)^{\alpha-2}\\
			&\leq C\brac{s\wedge t}^{2\beta-\alpha}\abs{s-t}^{\alpha-2}.
		\end{align*}
	\end{proof}
	\begin{Lemma}
		\label{mvtconsequence}
		Assuming the integrals appearing on the right hand sides of the
		equalities below are well defined, it holds that
		\begin{equation*}
			\lim_{d,p\to \infty} dp \sum_{k=0}^{d-1}\sum_{j=0}^{p-1}g\left(\frac{k}{d},\frac{j}{p}\right)\left[\int_{\frac{k}{d}}^{\frac{k+1}{d}} \int_{\frac{js}{p}}^{\frac{j+1}{p}}f(u,v)dudv  \right]^2
			=\int_{0}^{1}\int_{0}^{1}g(u,v)f(u,v)^2dudv
		\end{equation*}
		as well as 
		\begin{equation*}
			\lim_{p\to \infty} dp \sum_{k=0}^{d-1}\sum_{j=0}^{p-1}g\left(\frac{k}{d},\frac{j}{p}\right)\left[\int_{\frac{k}{d}}^{\frac{k+1}{d}} \int_{\frac{j}{p}}^{\frac{j+1}{p}}f(u,v)dudv\right]^2
			=d \sum_{k=0}^{d-1}\int_{0}^{1}g\brac{\frac{k}{d},v} \brac{\int_{\frac{k}{d}}^{\frac{k+1}{d}}f(u,v)du}^2dv.
		\end{equation*}
	\end{Lemma}
	\begin{proof}
		Both limits follow directly from the mean value theorem. 
	\end{proof}
	\begin{Lemma}
		\label{Lemmal2convergence}
		For any $1 \leq i,j \leq n$, the sequence $\left\{\widehat{W}_{ij}\brac{\floor{dx}} \colon d \in \N
		\right\}$ is Cauchy in $L^2 \left( \Omega \right)$.  
	\end{Lemma}
	\begin{proof}
		A sequence $\left\{ a_n \colon n\in \N \right\}$ in a Hilbert space
		$K$ is Cauchy in $K$ if and only if $\inner{a_n,a_m}_K \to C$ as $n,m
		\to \infty$, for some constant $C$ as
		\begin{align*}
			\norm{a_m-a_n}^2_K=\inner{a_m,a_m}_K+\inner{a_n,a_n}_K-2\inner{a_m,a_n}_K.
		\end{align*}
		Based on this observation, we only need to show that 
		\begin{align*}
			I=\lim_{d,p\to \infty}\E{\widehat{W}_{ij}\brac{\floor{dx}}\widehat{W}_{ij}(\floor{px})}<\infty
		\end{align*}
		for any $1 \leq i,j \leq n$. Thus, we use the first part of Lemma \ref{covarianceLemmarosenblatt} and \cite[Lemma 3.1]{HN} to write 
		\begin{align*}
			I&=\lim_{d,p\to \infty}\sum_{\substack{1\leq k\leq \floor{dx}\\
					1\leq j\leq \floor{px}-1}}\frac{\E{\brac{X_{\frac{k+1}{d}}-X_{\frac{k}{d}}}\brac{X_{\frac{l+1}{p}}-X_{\frac{l}{p}}} }^2}{\E{\brac{X_{\frac{k+1}{d}}-X_{\frac{k}{d}}}^2}\E{\brac{X_{\frac{l+1}{p}}-X_{\frac{l}{p}}}^2}}
			\\&= \frac{1}{4\lambda^2}\lim_{d,p\to \infty}\sum_{\substack{1\leq k\leq
					\floor{dx}\\ 1\leq j\leq \floor{px}-1}}dp \left(\frac{k}{d}\right)^{\alpha-2\beta}\left(\frac{j}{p}\right)^{\alpha-2\beta}\brac{\int_{\frac{j}{p}}^{\frac{j+1}{p}}\int_{\frac{k}{d}}^{\frac{k+1}{d}}\partial_{s,t}\E{X_s X_t}dsdt}^2\\
			&= \frac{1}{4\lambda^2}\int_0^x \int_0^x (st)^{\alpha-2\beta}\brac{\partial_{s,t}\E{X_s X_t}}^2dsdt\\
			&= \frac{1}{2\lambda^2}\int_0^x \int_0^t (st)^{\alpha-2\beta}\brac{\partial_{s,t}\E{X_s X_t}}^2dsdt\\
			&= \frac{1}{2\lambda^2}\int_0^x \int_0^\frac{t}{2} (st)^{\alpha-2\beta}\brac{\partial_{s,t}\E{X_s X_t}}^2dsdt+\frac{1}{2\lambda^2}\int_0^x \int_\frac{t}{2}^t (st)^{\alpha-2\beta}\brac{\partial_{s,t}\E{X_s X_t}}^2dsdt\\
			&= I_1+I_2.
		\end{align*}
		To handle $I_1$, we can use the second part of Lemma \ref{covarianceLemmarosenblatt},
		which implies
		\begin{align*}
			I_1 &\leq  C\int_0^x \int_0^\frac{t}{2} (st)^{\alpha-2\beta}\brac{\brac{s\wedge t}^{2\beta-\alpha}(s\vee t)^{\alpha-2}}^2dsdt\\
			&= C \int_0^x \int_0^\frac{t}{2} (st)^{\alpha-2\beta}\brac{s^{2\beta-\alpha}t^{\alpha-2}}^2dsdt\\
			&= C\int_0^x \int_0^\frac{t}{2} s^{2\beta-\alpha}t^{(\alpha-2\beta)+2(\alpha-2)}dsdt\\
			&= C\int_0^x t^{2\alpha-3}dsdt,
		\end{align*}
		which converges for $\alpha>\frac{3}{2}$. To deal with $I_2$, we
		appeal to Lemma \ref{covarianceLemmarosenblatt} once more to get
		\begin{align*}
			I_2 &= C\int_0^x \int_\frac{t}{2}^t (st)^{\alpha-2\beta}\brac{\partial_{s,t}\E{X_s X_t}}^2dsdt\\
			&\leq C\int_0^x \int_\frac{t}{2}^t (st)^{\alpha-2\beta}\brac{s^{2\beta-\alpha}(t-s)^{\alpha-2}}^2 dsdt\\
			&\leq C\int_0^x \int_\frac{t}{2}^t (t-s)^{2(\alpha-2)} dsdt\\
			&= C\int_0^x t^{2\alpha-3}dt,
		\end{align*}
		which is finite as well.
	\end{proof}

	\begin{acknow*}
		The authors would like to thank Guangqu Zheng for helpful and
		enjoyable discussions. The authors are also grateful
                to two anonymous referees for their insightful
                comments and remarks.
	\end{acknow*}

\end{document}